\documentclass[a4paper,12pt]{article}

\usepackage{amssymb,amsmath,amsthm,enumerate,cite}
\usepackage{graphicx,caption,subcaption,color}
\renewcommand{\epsilon}{\varepsilon}

\newcommand{\q}{\mathbf{q}}
\newcommand{\R}{\mathbb{R}}
\newcommand{\Z}{\mathbb{Z}}
\newcommand{\C}{\mathbb{C}}
\newcommand{\interior}{\operatorname{int}}
\renewcommand{\phi}{\varphi}
\newcommand{\pp}{\mathbf{p}}
\newcommand{\qq}{\mathbf{q}}
\newcommand{\xx}{\mathbf{x}}

\newcommand{\sg}{\mathrm{sgn}}

\newtheorem{thm}{Theorem}
\newtheorem{lemma}[thm]{Lemma}
\newtheorem{prop}[thm]{Proposition}
\newtheorem{cor}[thm]{Corollary}

\theoremstyle{definition}
\newtheorem{definition}[thm]{Definition}

\theoremstyle{remark}
\newtheorem{rmk}[thm]{Remark}

\title{Symplectic embeddings and the lagrangian bidisk}

\author{Vinicius Gripp Barros Ramos}
\AtEndDocument{\bigskip{\footnotesize
  \textsc{Instituto Nacional de Matem\'atica Pura e Aplicada, Rio de Janeiro, Brazil} \par  
  \textit{E-mail}:  \texttt{vgbramos@impa.br}
  }}
\date{}
\begin{document}

\maketitle

\abstract{In this paper we obtain sharp obstructions to the symplectic embedding of the lagrangian bidisk into four-dimensional balls, ellipsoids and symplectic polydisks. We prove, in fact, that the interior of the lagrangian bidisk is symplectomorphic to a concave toric domain using ideas that come from billiards on a round disk. In particular, we answer a question of Ostrover \cite{osticm}. We also obtain sharp obstructions to some embeddings of ellipsoids into the lagrangian bidisk.}
\section{Introduction}

Symplectic embedding questions have been central in the study of symplectic manifolds. The first of such questions was studied by Gromov in \cite{gromov}. After that, many techniques were created to deal with the questions of when symplectic embeddings exist. Symplectic capacities are one of such techniques and they provide an obstruction to the existence of a symplectic embedding. An interesting general question is whether a certain capacity is sharp for a certain embedding problem, i.e., whether the symplectic embedding exists if and only if this capacity does not give an obstruction to it.

ECH capacities are a sequence of capacities which are defined for four-dimensional symplectic manifolds \cite{qech}. For a symplectic manifold $(X^4,\omega)$, there is a sequence of real numbers:
\[0=c_0(X,\omega)< c_1(X,\omega)\le c_2(X,\omega) \le \dots \le\infty.\]
These numbers satisfy the following properties:
\begin{enumerate}[(i)]
\item If $a>0$, then $c_k(X,a\cdot \omega)=a \cdot c_k(X,\omega)$, for every $k$.
\item If $(X_1,\omega_1)$ symplectically embeds into $ (X_2,\omega_2)$, then \[c_k(X_1,\omega_1)\le c_k(X_2,\omega_2),\text{ for all $k$}.\]
 \item \begin{equation}\label{eq:disj}
c_k\left(\coprod_{i=1}^n (X_i,\omega_i)\right)=\max\left\{\sum_{i=1}^n c_{k_i}(X_i,\omega_i)\,\Bigg|\,k_1+\dots+k_n=k\right\}.
\end{equation}
\end{enumerate}
ECH capacities have been computed for many manifolds and they have been shown to be sharp for several symplectic embedding questions in dimension 4, see for example \cite{mcdel,ccfhr,dan2}. We recall that ECH capacities are said to be sharp for a certain embedding problem $(X_1,\omega_1)\hookrightarrow (X_2,\omega_2)$ if 
\[c_k(X_1,\omega_1)\le c_k(X_2,\omega_2),\,\forall \, k\;\Longrightarrow(X_1,\omega_1)\hookrightarrow (X_2,\omega_2).\]
In this paper the symbol $\hookrightarrow$ will always denote a symplectic embedding.
The goal of this paper is to prove some new results concerning the symplectic embeddings of the lagrangian bidisk into four-dimensional balls, ellipsoids and polydisks. In particular, this answers a question of Yaron Ostrover in \cite[\S 5]{osticm}.

We will now set up our notation. We will always consider $\R^4=\C^2$ with coordinates $(p_1,q_1,p_2,q_2)=(z_1,z_2)$ and its subsets  endowed with the symplectic form
\[\omega=\sum_{i=1}^2 dp_i\wedge dq_i.\] The main domain we are interested in is the lagrangian bidisk in $\R^4$, denoted by $P_L$, which is defined to be \[P_L=\{(p_1,q_1,p_2,q_2)\in \R^4\,|\,p_1^2+p_2^2\le 1,q_1^2+q_2^2\le 1\}.\]
We observe that the lagrangian product of any two disks is symplectomorphic to a multiple of $P_L$. We now define the ellipsoids $E(a,b)$ and the symplectic polydisks $P(a,b)$ as follows.
\begin{align*}
E(a,b)&=\left\{(z_1,z_2)\in\C^2\,\Bigg|\,\pi\left(\frac{|z_1|^2}{a}+\frac{|z_2|^2}{b}\right)\le 1\right\},\\
P(a,b)&=\left\{(z_1,z_2)\in\C^2\,|\,\pi|z_1|^2\le a,\pi|z_2|^2\le b\right\}.
\end{align*}
We denote the Euclidean ball of radius $\sqrt{a/\pi}$ by $B(a):=E(a,a)$.

The main result of this paper is the following theorem.
\begin{thm}\label{thm:emb}
ECH capacities give a sharp obstruction to symplectically embedding the interior of $P_L$ into balls, ellipsoids and symplectic polydisks. Moreover,
\begin{enumerate}[(a)]
\item $\interior(P_L) \hookrightarrow B(a)$ if and only if $a\ge3\sqrt{3}$,
\item $\interior(P_L) \hookrightarrow E(a,b)$ if and only if $\min(a,b)\ge4$ and $\max(a,b)\ge3\sqrt{3}$,
\item $\interior(P_L) \hookrightarrow P(a,b)$ if and only if $a,b\ge4$.
\end{enumerate}
\end{thm}

\begin{rmk}
Part (c) of Theorem \ref{thm:emb} was previously known by \cite{wong}. The proof combines the explicit construction of an embedding and Gromov's non-squeezing theorem.
\end{rmk}

\subsection{Toric domains}

Although understanding symplectic embeddings of four-dimensional symplectic manifolds in general is a very hard problem, many results are known for a certain class of manifolds called toric domains, which are constructed as follows. If $\Omega$ is a closed region in the first quadrant of $\R^2$, we define the \textit{toric domain} $X_{\Omega}\subset\C^2$ to be
\[X_{\Omega}=\{(z_1,z_2)\in\C^2;\pi(|z_1|^2,|z_2|^2)\in\Omega\}.\]
We endow $X_{\Omega}$ with the restriction of the standard symplectic form in $\C^2$.

The main result needed to prove Theorem \ref{thm:emb} is the following theorem.
\begin{thm}\label{thm:symp}
Let $X_0$ be the toric domain $X_{\Omega_0}$, where $\Omega_0$ is the region bounded by the coordinate axes and the curve parametrized by \begin{equation}\label{eq:param}\left(2\sin\left(\frac{\alpha}{2}\right)-\alpha\cos\left(\frac{\alpha}{2}\right),2\sin\left(\frac{\alpha}{2}\right)+(2\pi-\alpha)\cos\left(\frac{\alpha}{2}\right)\right),\qquad\alpha\in[0,2\pi].\end{equation} Then $\interior(P_L)$ and $\interior(X_0)$ are symplectomorphic.
\end{thm}
\begin{rmk}
The curve \eqref{eq:param} has some nice properties. For example, if we switch $\alpha$ by $2\pi -\alpha$, we deduce that this curve is symmetric with respect to the reflection about the line $y=x$. We also observe that $\frac{y'(\alpha)}{x'(\alpha)}=-\frac{2\pi-\alpha}{\alpha}.$
\end{rmk}

Two kinds of toric domains are of particular interest. Let $\Omega$ be the domain in the first quadrant of $\R^2$ bounded by the coordinate axes and a curve which is the union of the graph of a piecewise smooth non-increasing function $f:[0,a]\to [0,\infty)$ and the line segment $L$ connecting $(a,0)$ and $(a,f(a))$. We always assume that $f(0)>0$. If $f(a)=0$, we take $L=\emptyset$. We say that $X_{\Omega}$ is convex\footnote{This definition of convex toric domains is slightly less general than that given in \cite{dan2}, but it suffices for all of our applications.} is $f$ is a concave function, and that $X_{\Omega}$ is concave if $f$ is convex function and $L=\emptyset$, see Figure \ref{fig:tor}(a,b). We observe that ellipsoids and symplectic polydisks are convex toric domains and that ellipsoids are the only toric domains that are both convex and concave. Moreover the toric domain $X_0$ defined in Theorem \ref{thm:symp} is concave, see Figure \ref{fig:tor}(c). In \cite{dan2}, Cristofaro-Gardiner proved the following theorem.
\begin{thm}[Cristofaro-Gardiner]\label{thm:cc}
Let $X_{\Omega}$ and $X_{\Omega'}$ be concave and convex toric domains, respectively. Then ECH capacities give a sharp obstruction for embedding $\interior(X_{\Omega})$ into $X_{\Omega'}$.
\end{thm}
We observe that the first claim of Theorem \ref{thm:emb} follows immediately from Theorems \ref{thm:symp} and \ref{thm:cc}. 

\begin{figure}
\centering
\begin{subfigure}[b]{0.3\textwidth}
\centering
\def\svgwidth{0.7\textwidth}
\begingroup%
  \makeatletter%
  \providecommand\color[2][]{%
    \errmessage{(Inkscape) Color is used for the text in Inkscape, but the package 'color.sty' is not loaded}%
    \renewcommand\color[2][]{}%
  }%
  \providecommand\transparent[1]{%
    \errmessage{(Inkscape) Transparency is used (non-zero) for the text in Inkscape, but the package 'transparent.sty' is not loaded}%
    \renewcommand\transparent[1]{}%
  }%
  \providecommand\rotatebox[2]{#2}%
  \ifx\svgwidth\undefined%
    \setlength{\unitlength}{120bp}%
    \ifx\svgscale\undefined%
      \relax%
    \else%
      \setlength{\unitlength}{\unitlength * \real{\svgscale}}%
    \fi%
  \else%
    \setlength{\unitlength}{\svgwidth}%
  \fi%
  \global\let\svgwidth\undefined%
  \global\let\svgscale\undefined%
  \makeatother%
  \begin{picture}(1,1)%
    \put(0,0){\includegraphics[width=\unitlength]{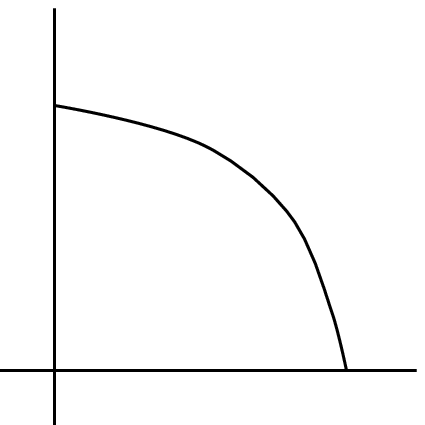}}%
    \put(0.21199407,0.2691331){\color[rgb]{0,0,0}\makebox(0,0)[lb]{\smash{$\Omega$}}}%
  \end{picture}%
\endgroup%
\caption{A convex toric domain}
\label{fig:conv}
\end{subfigure}
\quad
\begin{subfigure}[b]{0.3\textwidth}
\centering
\def\svgwidth{0.7\textwidth}
\begingroup%
  \makeatletter%
  \providecommand\color[2][]{%
    \errmessage{(Inkscape) Color is used for the text in Inkscape, but the package 'color.sty' is not loaded}%
    \renewcommand\color[2][]{}%
  }%
  \providecommand\transparent[1]{%
    \errmessage{(Inkscape) Transparency is used (non-zero) for the text in Inkscape, but the package 'transparent.sty' is not loaded}%
    \renewcommand\transparent[1]{}%
  }%
  \providecommand\rotatebox[2]{#2}%
  \ifx\svgwidth\undefined%
    \setlength{\unitlength}{120bp}%
    \ifx\svgscale\undefined%
      \relax%
    \else%
      \setlength{\unitlength}{\unitlength * \real{\svgscale}}%
    \fi%
  \else%
    \setlength{\unitlength}{\svgwidth}%
  \fi%
  \global\let\svgwidth\undefined%
  \global\let\svgscale\undefined%
  \makeatother%
  \begin{picture}(1,1)%
    \put(0,0){\includegraphics[width=\unitlength]{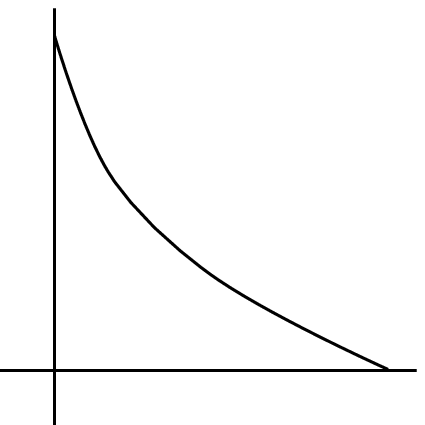}}%
    \put(0.21199407,0.2691331){\color[rgb]{0,0,0}\makebox(0,0)[lb]{\smash{$\Omega$}}}%
  \end{picture}%
\endgroup%
\caption{A concave toric domain}
\label{fig:conc}
\end{subfigure}
\quad
\begin{subfigure}[b]{0.3\textwidth}
\centering
\def\svgwidth{0.7\textwidth}
\begingroup%
  \makeatletter%
  \providecommand\color[2][]{%
    \errmessage{(Inkscape) Color is used for the text in Inkscape, but the package 'color.sty' is not loaded}%
    \renewcommand\color[2][]{}%
  }%
  \providecommand\transparent[1]{%
    \errmessage{(Inkscape) Transparency is used (non-zero) for the text in Inkscape, but the package 'transparent.sty' is not loaded}%
    \renewcommand\transparent[1]{}%
  }%
  \providecommand\rotatebox[2]{#2}%
  \ifx\svgwidth\undefined%
    \setlength{\unitlength}{164.38573303bp}%
    \ifx\svgscale\undefined%
      \relax%
    \else%
      \setlength{\unitlength}{\unitlength * \real{\svgscale}}%
    \fi%
  \else%
    \setlength{\unitlength}{\svgwidth}%
  \fi%
  \global\let\svgwidth\undefined%
  \global\let\svgscale\undefined%
  \makeatother%
  \begin{picture}(1,0.82722201)%
    \put(0,0){\includegraphics[width=\unitlength]{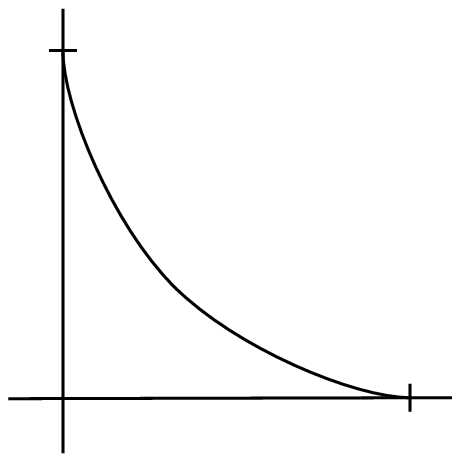}}%
    \put(0.20265663,0.26762511){\color[rgb]{0,0,0}\makebox(0,0)[lb]{\smash{$\Omega_0$}}}%
    \put(-0.06,0.87){\color[rgb]{0,0,0}\makebox(0,0)[lb]{\smash{$2\pi$}}}%
    \put(0.82,0.01){\color[rgb]{0,0,0}\makebox(0,0)[lb]{\smash{$2\pi$}}}%
  \end{picture}%
\endgroup%
\label{fig:om}
\caption{The region $\Omega_0$}
\end{subfigure}
\caption{Toric domains}
\label{fig:tor}
\end{figure}

\subsection{The boundary of $P_L$ and billiards}
The idea of the proof of Theorem \ref{thm:symp} is to put an appropriate Hamiltonian toric action on $\interior(P_L)$ and to compute the image of its moment map. We will now give a description of this action.

We would like first to define a toric action on $\partial P_L$ and then to extend it to all of $P_L$. We cannot do that because $\partial P_L$ is not smooth. But we can still get an idea of the actual definition which will be given in \S\ref{sub:def} by looking at $\partial P_L$. The 3-manifold $\partial P_L$ is a union of two solid tori $D^2_{\qq}\times S^1_{\pp}\cup S^1_{\qq}\times D^2_{\pp}$. The characteristic flow is generated by the vector field $V$ defined by:
\[V_\xx=\left\{\begin{aligned}p_i\sum_i\frac{\partial}{\partial q_i}, &\quad\text{if }\xx\in\interior(D^2_{\qq}\times S^1_{\pp}),\\
-q_i\sum_i\frac{\partial}{\partial p_i}, &\quad\text{if }\xx\in\interior(S^1_{\qq}\times D^2_{\pp}). \end{aligned}\right.\]
Note that we cannot extend $V$ continuously on $S^1_\qq\times S^1_\pp$. Even still, $V$ generates a continuous flow on $\partial P_L$ so that each time we hit the torus $S^1_\qq\times S^1_\pp$, we go from the interior of a solid torus to another, and so that there are two orbits contained in $S^1_\qq\times S^1_\pp$ which rotate along $S^1_\qq$ and $S^1_\pp$ with the same speed in the clockwise and counter-clockwise directions. We observe that if we look at the trajectory on $D^2_{\qq}\times S^1_\pp$ and project it to $D^2_\qq$, we obtain a billiard trajectory, as defined in \S\ref{sub:bil}, see Figure \ref{fig:bil}. We refer the reader to \cite{albmaz} and \cite{artost} for more details.

As we will see in \S\ref{sub:bil}, we can define a toric action on the set of points belonging to a billiard trajectory in $D^2_\qq$. Given such a point which is not on the two trajectories contained in $\partial D^2_\qq$, we would like to define two circle actions as follows. The first one is given by rotating $\qq$ and the corresponding $\pp$ by the same angle. The other one is given by moving along the billiard trajectory and rotating back by an angle whose proportion to the total angle spanned by the line segment is equal to the amount moved on it. These two actions correspond to translations in the toric coordinates $\varphi_2$ and $\varphi_1$, respectively, which will be defined in \S\ref{sub:bil}. 

\begin{figure}
\centering

\def\svgwidth{0.6\textwidth}
\begingroup%
  \makeatletter%
  \providecommand\color[2][]{%
    \errmessage{(Inkscape) Color is used for the text in Inkscape, but the package 'color.sty' is not loaded}%
    \renewcommand\color[2][]{}%
  }%
  \providecommand\transparent[1]{%
    \errmessage{(Inkscape) Transparency is used (non-zero) for the text in Inkscape, but the package 'transparent.sty' is not loaded}%
    \renewcommand\transparent[1]{}%
  }%
  \providecommand\rotatebox[2]{#2}%
  \ifx\svgwidth\undefined%
    \setlength{\unitlength}{291.68007813bp}%
    \ifx\svgscale\undefined%
      \relax%
    \else%
      \setlength{\unitlength}{\unitlength * \real{\svgscale}}%
    \fi%
  \else%
    \setlength{\unitlength}{\svgwidth}%
  \fi%
  \global\let\svgwidth\undefined%
  \global\let\svgscale\undefined%
  \makeatother%
  \begin{picture}(1,0.33220658)%
    \put(0,0){\includegraphics[width=\unitlength]{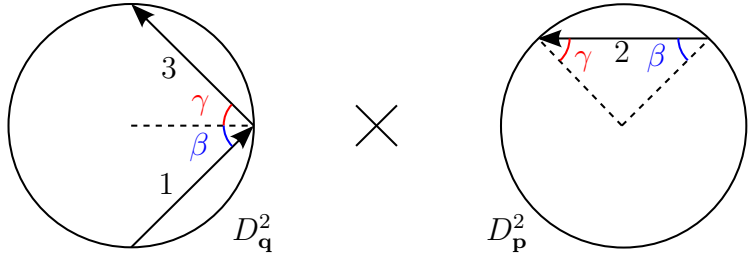}}%
    \put(0.24547443,0.13065101){\color[rgb]{0,0,0}\makebox(0,0)[lb]{\smash{$\color{blue}\beta$}}}%
    \put(0.24821716,0.19373382){\color[rgb]{0,0,0}\makebox(0,0)[lb]{\smash{$\color{red}\gamma$}}}%
    \put(0.86258891,0.25133117){\color[rgb]{0,0,0}\makebox(0,0)[lb]{\smash{$\color{blue}\beta$}}}%
    \put(0.76522196,0.25338822){\color[rgb]{0,0,0}\makebox(0,0)[lb]{\smash{$\color{red}\gamma$}}}%
    \put(0.20364778,0.07785343){\color[rgb]{0,0,0}\makebox(0,0)[lb]{\smash{1}}}%
    \put(0.20707619,0.23487478){\color[rgb]{0,0,0}\makebox(0,0)[lb]{\smash{3}}}%
    \put(0.81870521,0.25681663){\color[rgb]{0,0,0}\makebox(0,0)[lb]{\smash{2}}}%
    \put(0.30307178,0.0154563){\color[rgb]{0,0,0}\makebox(0,0)[lb]{\smash{$D^2_\qq$}}}%
    \put(0.64591316,0.0154563){\color[rgb]{0,0,0}\makebox(0,0)[lb]{\smash{$D^2_\pp$}}}%
  \end{picture}%
\endgroup%
\caption{The segments representing a billiard trajectory}
\label{fig:bil}
\end{figure}

\subsection{Ball packings and ECH capacities}\label{sec:pack}

ECH capacities of a concave toric domain can be computed using an appropiate ball packing, as explained in \cite{ccfhr}. We now recall this construction and compute the first two ECH capacities of $X_0$.

Let $X_\Omega$ be a concave toric domain. The {\em weight expansion} of $\Omega$ is the multiset $w(\Omega)$ defined as follows. For a triangle $T$ with vertices $(0,0)$, $(a,0)$ and $(0,a)$, we define $w(T)=\{a\}$. If $\Omega=\emptyset$, we let $w(\Omega)=\emptyset$. Let $T_1$ be the largest triangle contained in $\Omega$. Then $\Omega\setminus T_1=\Omega_1'\sqcup\Omega_2'$, where $\Omega_1'$ and $\Omega_2'$ could be empty. We translate the closures of $\Omega_1'$ and $\Omega_2'$ so that the obtuse corners are at the origin and we multiply these regions by the matrices $\left[\begin{array}{cc}1&1\\0&1\end{array}\right]$ and $\left[\begin{array}{cc}1 &0\\1&1\end{array}\right]$, respectively, obtaining two regions that we call  $\Omega_1$ and $\Omega_2$, respectively. In particular, $X_{\Omega_1}$ and $X_{\Omega_2}$ are also concave toric domains. Assuming that $w(\Omega_1)$ and $w(\Omega_2)$ are defined, we let \[w(\Omega):=w(T_1)\cup w(\Omega_1)\cup w(\Omega_2).\]
Here we consider the union with repetition. We now proceed by induction to define $w(\Omega_1)$ and $w(\Omega_2)$ in terms of triangles and smaller regions. This process is infinite, unless $\Omega$ is bounded by the graph of a piecewise linear function whose slopes are all rational. 

We now define the {\em weight sequence} $w_1\ge w_2 \ge \dots$ to be the non-increasing ordering of the elements of $w(\Omega)$.
As explained in \cite{ccfhr} and reviewed in \S\ref{sec:emb}, for every $\epsilon>0$, there exists a symplectic embedding \begin{equation}\label{eq:pack}\coprod_{i=1}^{\infty} B(w_i)\hookrightarrow (1+\epsilon)X_\Omega.\end{equation} Therefore for every $k\in\mathbb{N}$,
\begin{equation}\label{eq:lim}\lim_{l\to\infty}c_k\left(\coprod_{i=1}^{l} B(w_i)\right)\le c_k(X_\Omega).\end{equation}
\begin{rmk}
Since $(w_i)_{i\ge 1}$ is a non-increasing sequence, it follows from \eqref{eq:disj} that
\[c_k\left(\coprod_{i=1}^{l} B(w_i)\right)=c_k\left(\coprod_{i=1}^{k} B(w_i)\right),\qquad \text{for }l\ge k.\]
Therefore the limit in \eqref{eq:lim} equals $c_k\left(\coprod_{i=1}^{k} B(w_i)\right)$.
\end{rmk}
The main theorem of \cite{ccfhr} is the following.
\begin{thm}[Choi, Cristofaro-Gardiner, Frenkel, Hutchings, Ramos \cite{ccfhr}]\label{thm:ccfhr}
Let $X_{\Omega}$ be a concave toric domain and let $w_1\ge w_2\ge w_3\ge \dots$ be the weight sequence of $\Omega$. Then for every $k\in\mathbb{N}$,
\begin{equation}
c_k\left(\coprod_{i=1}^{k} B(w_i)\right)=c_k(X_\Omega).\label{eq:ccfhr}
\end{equation}
\end{thm}

In \S\ref{sec:emb}, we will show that the first two weights of $\Omega_0$ are:
\begin{equation}\label{eq:w}w_1=4,\quad w_2=3\sqrt{3}-4.\end{equation}
From \eqref{eq:ccfhr} and \eqref{eq:w} it follows that:
\begin{equation}\label{eq:ckx}
c_1(X_0)=4,\quad c_2(X_0)=3\sqrt{3}.
\end{equation}
We note that \eqref{eq:ckx} is enough to obtain the 'only if' parts of Theorem \ref{thm:emb} as we explain below.

\subsection{Proof of Theorem \ref{thm:emb}}\label{sec:thm}

The last ingredient of the proof of Theorem \ref{thm:emb} is the following proposition.

\begin{prop}\label{prop:emb2}
There exists a symplectic embedding $\interior(X_0)\hookrightarrow E(4,3\sqrt{3})$.
\end{prop}

Assuming Theorem \ref{thm:symp} and Proposition \ref{prop:emb2}, we can now prove Theorem \ref{thm:emb}.

\begin{proof}[Proof of Theorem \ref{thm:emb}]
By Theorem \ref{thm:symp}, we can substitute in Theorem \ref{thm:emb} $\interior(P_L)$ by the concave toric domain $\interior(X_0)$.
\newline

(a) First let us assume that $\interior(X_0)$ symplectically embeds into $B(c)$. Then
\[3\sqrt{3}=c_2(X_0)\le c_2(B(c))=c.\]
Conversely, $E(4,3\sqrt{3})\subset B(3\sqrt{3})\subset B(c)$, for all $c\ge 3\sqrt{3}$. So, by Proposition \ref{prop:emb2}, $\interior(X_0)$ symplectically embeds into $B(c)$ for all $c\ge 3\sqrt{3}$.\newline

(b) Assume that $\interior(X_0)\hookrightarrow E(a,b)$, where $a\le b$. We recall that $c_1(E(a,b))=a$ and $c_2(E(a,b))=\min(2a,b)$. Hence
\begin{align*}4&=c_1(X_0)\le c_1(E(a,b))=a,\\
3\sqrt{3}&=c_2(X_0)\le c_2(E(a,b))\le b.
\end{align*}
The converse is a direct consequence of Proposition \ref{prop:emb2}.\newline

(c) Assume that $\interior(X_0)\hookrightarrow P(a,b)$, where $a\le b$. Again we have \[4=c_1(X_0)\le c_1(P(a,b))=a.\] For the converse, we can construct an explicit embedding $\interior(P_L)\hookrightarrow P(4,4)$ by \[(p_1,q_1,p_2,q_2)\mapsto \left(\sqrt{\frac{2(p_1+1)}{\pi}}e^{i\pi(q_1+1)},\sqrt{\frac{2(p_2+1)}{\pi}}e^{i\pi(q_2+1)}\right).\] 

\end{proof}

\subsection{A converse question}

We may also ask a converse question, namely, when ellipsoids embed into the lagrangian bidisk. Although this question is still open in general, we can answer it in two cases.

\begin{cor}\label{thm:emb2}
Let $a\in\{1,2\}$. Then ECH capacities give a sharp obstruction to symplectically embedding $\interior(E(ab,b))$ into $\interior(P_L)$. In particular $\interior(E(ab,b))\hookrightarrow \interior(P_L)$ if, and only if, $\interior(E(ab,b))\subset \interior(X_0)$.
\end{cor}
\begin{proof}
We assume that $\interior(E(ab,b))\hookrightarrow \interior(P_L)$. We consider first the case $a=1$. Then it follows from \eqref{eq:ckx} that $b=c_1(B(b))\le c_1(P_L)=c_1(X_0)=4$. So \[\interior(B(b))\subset \interior(B(4))\subset \interior(X_0).\] The second inclusion above follows from a simple calculation, see Figure \ref{fig:w}(a). Now suppose that $a=2$. From \eqref{eq:ckx} we obtain $2b=c_2(E(2b,b))\le c_2(X_0)=3\sqrt{3}$. So $\interior(E(2b,b))\subset \interior(E(3\sqrt{3},3\sqrt{3}/2))$. We also observe that the line $x+2y=3\sqrt{3}$ is tangent to the curve \eqref{eq:param} and hence $\interior(E(3\sqrt{3},3\sqrt{3}/2))\subset \interior(X_0)$. Therefore $\interior(E(2b,b))\subset\interior( X_0)$.
\end{proof}

\begin{rmk}
Corollary \ref{thm:emb2} does not hold for $a\ge 5$. In fact, one can construct better embeddings than the inclusion by symplectic folding, see \cite{schbook}. We do not know whether there are better embeddings than the inclusion for $a=3,4$.
\end{rmk}

\begin{rmk}
Gutt and Hutchings have recently announced a result that implies that $P(a,a)\hookrightarrow\interior(P_L)$ if, and only if, $a\le 2$.
\end{rmk}

\subsection{Outline of the paper}
The rest of this paper is organized as follows. In \S\ref{sec:conc}, we prove that $\interior(P_L)$ is symplectomorphic to the interior of a concave toric domain, namely $X_0$, thus proving Theorem \ref{thm:symp}. In \S\ref{sec:emb}, we prove Proposition \ref{prop:emb2}, that is, we show that $\interior(X_0)$ embeds into $E(4,3\sqrt{3})$. As explained in \S\ref{sec:thm}, this concludes the proof of Theorem \ref{thm:emb}.

\paragraph{Acknowledgments.}
I would like to thank Felix Schlenk for asking the question that inspired this paper and Michael Hutchings for helpful discussions. During the course of this work, I was supported by the European Research Council Grant Geodycon and a grant of the French region Pays de la Loire. I would also like to thank the anonymous referees for very helpful comments and suggestions.

\section{Concave toric domains and the lagrangian bidisk}\label{sec:conc}

\subsection{Outline}
In this section, we will prove Theorem \ref{thm:symp}. We recall that $P_L$ is the product of two lagrangian disks and $X_0$ is the concave toric domain $X_{\Omega_0}$ where $\Omega_0$ is the region bounded by the coordinate axes and the curve \eqref{eq:param}.
We will prove that $\interior(P_L)$ and $\interior(X_0)$ are symplectomorphic.

The idea is to exhaust $\interior(P_L)$ by domains $P_{\epsilon}$ which are endowed with a Hamiltonian toric action whose moment image converges to $\Omega_0$. In other words, for each $0<\epsilon<1$, we will construct symplectic manifolds $P_\epsilon\subset \interior(P_L)$ and toric domains $X_{\Omega_\epsilon}\subset \interior(X_0)$ such that for each $\epsilon$, the domains $P_\epsilon$ and $X_{\epsilon}$ are symplectomorphic and \[\bigcup_\epsilon P_\epsilon=\interior(P_L)\qquad\text{ and }\qquad\bigcup_\epsilon X_{\epsilon}=\interior(X_0).\]

The definition of $P_{\epsilon}$ is relatively simple and uses an idea from \cite{bengia}, which was also used in \cite{albmaz}. Let $U:[0,1)\to \mathbb{R}_+$ be a smooth function such that:
\begin{itemize}
\item $U(0)<1$,
\item $U^{(j)}(r)>0$, for $j=1,2$ and for all $r>0$,
\item $U(r)\to \infty$ as $r\to 1$.
\end{itemize}
To simplify the notation, we will denote a point in $\R^4$ by $(\qq,\pp)$, where $\qq=(q_1,q_2)$ and $\pp=(p_1,p_2)$ although the orientation of $\R^4$ is still given by $\omega\wedge \omega$.
For $0<\epsilon<1$, let $H_\epsilon(\qq,\pp)=\frac{1}{2}\left(|\pp|^2+\epsilon U(|\qq|^2)\right)$ and \[P_{\epsilon}=\big\{(\qq,\pp)\in P_L\;\big|\;H_{\epsilon}(\qq,\pp)\le \frac{1}{2}\big\}.\]
We observe that $P_\epsilon$ is a Liouville domain with Liouville form
\[\lambda=\frac{1}{2}\sum_{i=1}^2 (p_idq_i-q_idp_i).\]
We also note that for $\epsilon<\epsilon'$, we have $P_{\epsilon'}\subset \interior(P_{\epsilon})$, and that $\bigcup_{\epsilon} P_{\epsilon}=\interior(P_L)$.

The definition of $X_{\epsilon}$ is more complicated and it will be given in \S\ref{sub:ext}. The idea is as follows. We first define a function $v:\partial P_\epsilon \to [-M,M]$, for some $M\in\R$ such that $v^{-1}(\{-M,M\})$ is a Hopf link. Then we define toric coordinates $(\phi_1,\phi_2)$ in the complement of this Hopf link. Finally we show that $\phi_1$ and $\phi_2$ extend to the different components of this link and that there exist functions $\rho_1$ and $\rho_2$ defined on $\partial P_\epsilon$ such that
\[\lambda|_{\partial P_\epsilon}=\rho_1 \,d\phi_1+\rho_2\, d\phi_2.\] We define $X_\epsilon$ to be the toric domain $X_{\Omega_\epsilon}$ where $\Omega_\epsilon$ is the region bounded by the coordinate axes and the image of $(\rho_1,\rho_2)$.

\subsection{Billiards}\label{sub:bil}
We will now give an idea of how to define the toric coordinates on $\partial P_\epsilon$ by taking the limit $\epsilon\to 0$. 
We will explain the heuristics in this subsection and give the actual definition that we will use to prove Theorem \ref{thm:symp} in \S\ref{sub:def}.

The Liouville form $\lambda$ is a contact form on each $\partial P_\epsilon$ and its Reeb flow is parallel to the Hamiltonian vector field
\[X_{H_\epsilon}=\sum_{i=1}^2 p_i\frac{\partial}{\partial q_i}-\epsilon U'(|\qq|^2)\sum_{i=1}^2 q_i\frac{\partial}{\partial p_i}.\]
If $(\qq(t),\pp(t))$ is a trajectory of this flow, then $\pp(t)=\dot{\qq}(t)$. So $(\qq(t),\pp(t))$ is determined by the curve $\qq(t)$ which satisfies the equation $\ddot{\q}(t)=-\epsilon U'(|\q(t)|^2)\q(t)$.
As explained in \cite{bengia} and in \cite{albmaz}, a sequence of solutions to this equation with $\epsilon\to 0$ and bounded energy admits a subsequence which converges to a closed billiard trajectory in a suitable topology.

To get an idea of what is happening, we observe that for very small $\epsilon$, the acceleration $\ddot{\q}$ is very close to $0$, except in a neighborhood of $\partial D^2$. So $\qq$ is very close to a line segment away from $\partial D^2$ and it bends sharply near $\partial D^2$. At a point of maximum $t_0$ of $|\qq(t)|$, we have $\q(t_0)\cdot\dot{\qq}(t_0)=0$. It follows from the proof of Lemma \ref{lem:ind} below that $\qq(t)$ is symmetric with respect to the reflection about the line spanned by $\qq(t_0)$. Moreover 
\begin{align*}0&=\int_{t_0-\delta}^{t_0+\delta}\left\langle \ddot{\q}(t)+\epsilon U'(|\q(t)|^2)\q(t),\dot{\qq}(t_0)\right\rangle\,dt\\&=\left\langle \dot{\q}(t_0+\delta),\dot{\qq}(t_0)\right\rangle-\left\langle \dot{\q}(t_0-\delta),\dot{\qq}(t_0)\right\rangle+\epsilon \int_{t_0-\delta}^{t_0+\delta}\left\langle U'(|\q(t)|^2)\q(t),\dot{\qq}(t_0)\right\rangle\,dt.\end{align*}
So if we take a family of curves $\qq_\epsilon$ with constant $Z:=\dot{\qq}_\epsilon(t_0)$ such that the curve $\qq_\epsilon(t)$ for $t\in[t_0-\delta,t_0)$ converges to a line segment of direction $\dot{\q}(t_0-\delta)$, then the limiting curve $\underline{\q}(t)$ for $t\in(t_0,t_0+\delta]$ will also be a line segment and
\[\left\langle \dot{\underline{\q}}(t_0+\delta),Z\right\rangle=\left\langle \dot{\underline{\q}}(t_0-\delta),Z\right\rangle.\]
So $\underline{\qq}(t)$ is what we call a billiard trajectory.

A {\em billiard trajectory} is a curve $\qq(t)$ in $D^2\subset \R^2$ which is piecewise smooth and satisfies:
\begin{itemize}
\item $\ddot{\q}(t)=0$ and $|\dot{\q}(t)|=1$ whenever $\q$ is smooth at $t$.
\item If $\q$ is not smooth at $t_0$, then $\q(t_0)\in\partial D^2$ and 
\[\lim_{t\to t_0^-}\left\langle\dot{\q}(t),\q(t_0)\right\rangle=-\lim_{t\to t_0^+}\left\langle\dot{\q}(t),\q(t_0)\right\rangle.\]
\end{itemize} 

Let $Y$ be the space of points $(\qq,\pp)$ that belong to a billiard trajectory of $D^2$. Here $\pp$ is the velocity of the billiard trajectory at $\qq$. A natural pair of commuting independent Hamiltonians for the billiard flow on the disk is $(H_0,v)$ where $H_0(\qq,\pp)=\frac{1}{2}|\pp|^2$ and $v(\qq,\pp)=\qq\times\pp$ is the angular momentum. But the induced action is not toric. In fact the vector field $X_{H_0}$ induces an $\R$-action which is usually not periodic. We can use $v$ to produce a pair of Hamiltonians which generate a toric action. We do that indirectly by defining explicit action-angle coordinates as we explain below.
  
We let $\alpha(\qq,\pp)=2\arccos(\qq\times \pp)\in (0,2\pi)$. Let $L$ be the set of the points in $Y$ corresponding to the oriented line segment from $\qq_0\in\partial D^2$ to $\qq_1\in\partial D^2$. For $(\qq,\pp)\in L$, we define $s(\qq,\pp)$ to be the ratio $|\overrightarrow{\qq_0\qq}|/|\overrightarrow{\qq_0\qq_1}|$. It follows from a simple calculation that
\begin{equation*}
s(\qq,\pp)=\frac{\qq\cdot \pp+\sin\left(\frac{\alpha(\qq,\pp)}{2}\right)}{2\sin\left(\frac{\alpha(\qq,\pp)}{2}\right)}.
\end{equation*} 
We also define $\psi(\qq,\pp)=\arg(\qq_0)+s(\qq,\pp)\alpha(\qq,\pp)\in\R/2\pi\Z$.

We can see $Y$ as a subset of $\partial P_\epsilon$. Under this inclusion, it follows from a calculation using the definitions above that
\[\lambda=\left(2\sin\left(\frac{\alpha}{2}\right)-\alpha \cos\left(\frac{\alpha}{2}\right)\right)ds+\cos\left(\frac{\alpha}{2}\right)d\psi.\]
In order to obtain an actual toric domain, we need to perform a change of variables:
 \begin{align*}
\phi_1(\xx)&= s(\xx)-\frac{\psi(\xx)}{2\pi}\in\R/\Z,\\
\phi_2(\xx)&= \frac{\psi(\xx)}{2\pi}\in \R/\Z.
\end{align*}
So
\begin{equation}\label{eq:lam}\lambda=\left(2\sin\left(\frac{\alpha}{2}\right)-\alpha \cos\left(\frac{\alpha}{2}\right)\right)d\phi_1+\left(2\sin\left(\frac{\alpha}{2}\right)+(2\pi-\alpha)\cos\left(\frac{\alpha}{2}\right)\right)d\phi_2.\end{equation}

In the following sections, we will make these ideas precise and explain how an equation such as \eqref{eq:lam} implies that $\interior(P_L)$ is symplectomorphic to a toric domain.

\subsection{The toric coordinates}\label{sub:def}

We now fix $0<\epsilon<1$ and we let $Y=\partial P_{\epsilon}$ endowed with the contact form
\[\lambda=\frac{1}{2}\sum_{i=1}^2 (p_idq_i-q_idp_i).\]
For $(\qq,\pp)\in Y$, we define $v(\qq,\pp)=\qq\times \pp\in\R$, where $\times$ denotes the two-dimensional cross-product. We observe that $v$ is constant along the Reeb trajectories.
\begin{lemma}
The function $v$ takes values in $[-M,M]$ for some $M$. Moreover $v^{-1}(M)$ and $v^{-1}(-M)$ are circles.
\end{lemma}
\begin{proof}
The Reeb flow is parallel to the vector field
\[V=\sum_{i=1}^2 p_i\frac{\partial}{\partial q_i}-\epsilon U'(|\qq|^2)\sum_{i=1}^2 q_i\frac{\partial}{\partial p_i}.\]
An integral curve of $V$ is a solution $(\qq(t),\pp(t))$ to the system of differential equations:
\begin{equation}
\left\{\begin{aligned}
\dot{\qq}(t)&=\pp(t)\\
\dot{\pp}(t)&=-\epsilon U'(|\qq|^2) \qq(t).
\end{aligned}\right.\label{eq:sys}
\end{equation}
In particular, a solution to \eqref{eq:sys} is determined by its projection $\qq(t)$, which satisfies
 \begin{equation}\label{eq:ddot}
\ddot{\qq}(t)=-\epsilon U'(|\qq|^2) \qq(t).
\end{equation}
Moreover, when specifying the initial conditions $(\qq(0),\dot{\qq}(0))$ to \eqref{eq:ddot}, it is enough to give the direction of $\dot{\qq}(0)$ since its length is determined by the fact that $(\qq(0),\dot{\qq}(0))\in Y$.

Let $(\qq(t),\pp(t))$ be a parametrization of an integral curve of $V$. We observe that 
\begin{align}
\frac{d}{dt}\Big(|\qq(t)|^2\Big)&=2 \pp(t)\cdot \qq(t),\label{eq:dot1}\\
\frac{d}{dt}\Big(\qq(t)\cdot \pp(t)\Big)&=|\pp(t)|^2-\epsilon U'(|\qq(t)|^2)|\qq(t)|^2.\label{eq:dot2}
\end{align}
It follows from \eqref{eq:dot1} and \eqref{eq:dot2} that $|\qq(t)|^2$ always has a maximum.  We note that the points of local extrema of $|\q(t)|$ are the same as the ones of $|\q(t)|^2$, but if $v(\q(t),\pp(t))=0$, then $|\q(t)|$ is not smooth for $t$ such that $\q(t)=0$.

If $t_0$ is a point of maximum or minimum of $|\qq(t)|^2$, then for every $t$, \[v(\qq(t),\pp(t))=\qq(t_0)\times\pp(t_0)=\delta|\qq(t_0)|\cdot|\pp(t_0)|=r_0\sqrt{1-\epsilon U(r_0^2)},\]
where $r_0=\delta|\qq(t_0)|$ and $\delta$ is the sign of $\qq(t_0)\times\pp(t_0)$. For $r\in [-\sqrt{U^{-1}(1/\epsilon)},\sqrt{U^{-1}(1/\epsilon)}]$, let $f(r)=r\sqrt{1-\epsilon U(r^2)}$. It follows from our choice of $U$ that $f$ is an odd function and that it has exactly two critical points $\bar{r}>0$ and $-\bar{r}<0$. Moreover $\bar{r}$ and $-\bar{r}$ are the points of global maximum and minimum, respectively. Let $M=f(\bar{r})$. So $f$ takes values in $[-M,M]$.

Let $C_{\pm}=v^{-1}(\pm M)$. We will show that $C_+$ and $C_-$ are circles. Let $(\qq(t),\pp(t))$ be the integral trajectory of $V$ such that $v(\qq(t),\pp(t))=M$ and let $t_0$ be a point of maximum of $|\qq(t)|^2$. So $|\qq(t_0)|=\bar{r}$. From $f'(\bar{r})=0$ it follows that
\[0=1-\epsilon U(|\qq(t_0)|^2)-\epsilon U'(|\qq(t_0)|^2)|\qq(t_0)|^2=|\pp(t_0)|^2-\epsilon U'(|\qq(t_0)|^2)|\qq(t_0)|^2.\] So $\pp(t_0)=\sqrt{\epsilon U'(|\qq(t_0)|^2)} i\cdot \qq(t_0)$, where $\cdot$ denotes complex multiplication in the plane $(q_1,q_2)$. Now let \begin{equation}\label{eq:circ}\underline{\qq}(t)=e^{ i\sqrt{\epsilon U'(|\qq(t_0)|^2)} (t-t_0)}\qq(t_0).\end{equation} Then $\underline{\qq}$ satisfies~\eqref{eq:ddot} and $(\underline{\qq}(t_0),\underline{\dot{\qq}}(t_0))=(\qq(t_0),\pp(t_0))$. By the uniqueness of solutions of differential equations, $\underline{\qq}(t)=\qq(t)$. So $C_+\subset Y$ is a circle. Analogously, we can show that $C_-$ is a circle.

\end{proof}

Let $C_{\pm}$ be the circles defined above and let $\widehat{Y}=Y\setminus (C_+\cup C_-)$. We will show that $v|_{\widehat{Y}}$ is a torus bundle and we will define a trivialization $(\phi_1,\phi_2):\widehat{Y}\to T^2$. In other words, we will construct a diffeomorphism $\widehat{Y}\cong (-M,M)\times T^2$. Before doing that, we will prove a lemma that will be necessary for the definition of the functions $\phi_1$ and $\phi_2$.

\begin{lemma}\label{lem:ind}
Let $(\qq(t),\pp(t))$ be a Reeb trajectory.
 \begin{enumerate}[(a)]
\item If $t_0<t_1$ are two consecutive points of maximum of  $|\qq(t)|$, then the differences $t_1-t_0\in\R$ and $\arg(\qq(t_1))-\arg(\qq(t_0))\in \R/2\pi\Z$ are independent of the choice of the pair $t_0,t_1$.

\item The differences in (a) depend only on the value of $v(\qq(t),\pp(t))$.
\end{enumerate}

\end{lemma}

\begin{proof}
(a) Let $(\qq(t),\pp(t))$ be a Reeb trajectory and let $(\tilde{\qq}(\tau),\tilde{\pp}(\tau))$ be a parametrization of the same curve, but now as an integral curve of $V$, i.e.,
\begin{equation*}\label{eq:rep}
(\tilde{\qq}(\tau),\tilde{\pp}(\tau))=(\qq(t(\tau)),\pp(t(\tau))),
\end{equation*}
for some smooth function $t(\tau)$. By a simple computation, we obtain:
\begin{equation*}
t'(\tau)=\frac{1}{2}\Big(1-\epsilon U(|\q(t(\tau))|^2)+\epsilon U'(|\q(t(\tau))|^2)|\q(t(\tau))|^2\Big)=K(|\q(t(\tau))|^2),
\end{equation*}
where $K(u)=\frac{1}{2}(1-\epsilon U(u)+\epsilon U'(u) u)$. So
\begin{equation}\label{eq:K2}
\pp(t)=K(|\qq(t)|^2)\dot{\qq}(t).
\end{equation}
We write $\q(t)$ in polar coordinates $\q(t)=r(t)e^{i\theta(t)}$. It follows from \eqref{eq:K2} that \eqref{eq:ddot} is equivalent to the following system of equations:
\begin{equation}
\left\{\begin{aligned}
K(r^2)^2(\ddot{r}-r(\dot{\theta})^2)+2K(r^2)K'(r^2)r(\dot{r})^2&=-\epsilon U'(r^2) r.\\
K(r^2)^2(2\dot{r}\dot{\theta}+r\ddot{\theta})+2K(r^2)K'(r^2)r^2\dot{r}\dot{\theta}&=0.
\end{aligned}\right.
\label{eq:polar}
\end{equation}
Now let $t_0<t_1<t_2$ be three consecutive points of maximum of $r(t)$. By a translation of time, we can assume without loss of generality that $t_0=0$. We now let $\underline{r}(t)=r(2t_1-t)$ and $\underline{\theta}(t)=2\theta(t_1)-\theta(2t_1-t)$. We observe that $(\underline{r},\underline{\theta})$ satisfies \eqref{eq:polar}. Moreover $$\underline{r}(t_1)=r(t_1),\quad \underline{\theta}(t_1)=\theta(t_1),\quad \dot{\underline{r}}(t_1)=\dot{r}(t_1)=0,\quad\dot{\underline{\theta}}(t_1)=\dot{\theta}(t_1).$$ By the uniqueness of solutions of differential equations, we conclude that $\underline{r}(t)= r(t)$ and $\underline{\theta}(t)= \theta(t)$. So
\begin{align}
r(t)&=r(2t_1-t),\nonumber \\
\theta(t)&=2\theta(t_1)-\theta(2t_1-t). \label{eq:theta}
\end{align}
Now, since $\dot{r}(t)=-\dot{r}(2t_1-t)$ and $\ddot{r}(t)=\ddot{r}(2t_1-t)$ and since there are no points of maximum of $r(t)$ in $(0,t_1)$, it follows that $2t_1$ is a point of maximum of $r(t)$ and that there are no other points of maximum in the interval $(t_1,2t_1)$. So $t_2-t_1=t_1$ and $\theta(t_2)-\theta(t_1)=\theta(t_1)-\theta(0)$. By induction, we conclude that the difference between any two consecutive points of maximum of $r(t)$ is always $t_1$ and that the difference between their $\theta$-values is $\theta(t_1)-\theta(0)$.
\newline

 (b) We first claim that if $(\qq_1,\pp_1),(\qq_2,\pp_2)\in\widehat{Y}$ are such that $v(\qq_1,\pp_1)=v(\qq_2,\pp_2)$, then $(A\cdot \qq_1,A\cdot \pp_1)$ is on the same Reeb trajectory as $(\qq_2,\pp_2)$ for some $A\in SO(2,\R)$.
Let $(\qq_1(t),\pp_1(t))$ and $(\qq_2(t),\pp_2(t))$ be Reeb trajectories going through $(\qq_1,\pp_1)$ and $(\qq_2,\pp_2)$ at time $0$, respectively. We can assume without loss of generality that $t=0$ is a point of minimum of $|\qq_1(t)|^2$ and $|\qq_2(t)|^2$.
We assumed that $\qq_1\times\pp_1=\qq_2\times\pp_2$. It follows from~\eqref{eq:dot1} that $\qq_j\cdot\pp_j=0$ for $j=1,2$. For $-\sqrt{U^{-1}(1/\epsilon)}\le r\le \sqrt{U^{-1}(1/\epsilon)}$, we recall that $f(r)=r\sqrt{1-\epsilon U(r^2)}$.
Since $(\qq_j,\pp_j)\in \widehat{Y}$, we have
\begin{equation}\label{eq:equal}f(\delta |\qq_1|)=\qq_1\times\pp_1=\qq_2\times\pp_2=f(\delta |\qq_2|),\end{equation}
where $\delta$ is the sign of $\qq_1\times\pp_1$.
We observe that \begin{equation}\label{eq:fder}f'(r)=\frac{1-\epsilon U(r^2)-\epsilon U'(r^2) r^2}{\sqrt{1-\epsilon U(r^2)}}.\end{equation}
Let $\bar{r}$ be the unique positive critical point of $f$. Then $-\bar{r}<r< \bar{r}$ if, and only if, $f'(r)>0$. In particular $f$ is injective on $(-\bar{r},\bar{r})$.
Since $\frac{d^2}{dt^2}|_{t=0}(|\qq_j(t)|^2)>0$, it follows from~\eqref{eq:dot1} and~\eqref{eq:dot2} that $\delta|\qq_j|\in(-\bar{r},\bar{r})$. So~\eqref{eq:equal} implies that $|\qq_1|=|\qq_2|$. Thus there exists $A\in SO(2)$ such that $A\cdot \qq_1=\qq_2$. Since $\qq_j\cdot\pp_j=0$ it follows that $\pp_j$ is a positive multiple of the $\delta\pi/2$ rotation of $\qq_j$. So $A\cdot\pp_1=\pp_2$. 
By the uniqueness of solutions of differential equations, it follows that \[(A\cdot \qq_1(t),A\cdot \pp_1(t))=(\qq_2(t),\pp_2(t)).\] So the differences in (a) are equal for the curves $(\qq_1(t),\pp_1(t))$ and $(\qq_2(t),\pp_2(t))$. Therefore these differences only depend on the value of the function $v$.

\end{proof}

 We will now define functions $s:\widehat{Y}\to \R/\Z$ and $\psi:\widehat{Y}\to \R/2\pi\Z$ as follows. Let $v\in(-M,M)$ and let $(\qq(t),\pp(t))$ be a Reeb trajectory with $v=v(\qq(t),\pp(t))$. If $t_0<t_1$ are two consecutive points of maximum of $|\qq(t)|$, we let $G(v)=t_1-t_0\in\R$ and $\underline{\alpha}(v(\qq(t),\pp(t)))=\arg(\qq(t_1))-\arg(\qq(t_0))\in \R/2\pi\Z.$ It follows from Lemma \ref{lem:ind} that $G(v)$ and $\underline{\alpha}(v)$ are well-defined. We now let $\alpha(v)\in\R$ be the continuous lift of $\underline{\alpha}(v)$ satisfying $\alpha(0)=\pi$.

\begin{definition} 
For a Reeb trajectory $(\qq(t),\pp(t))$ and a point of maximum $t_0$ of $|\q(t)|$, we let \begin{align*}
\tilde{s}(\qq(t),\pp(t))&=\frac{t-t_0}{G(v(\qq(t),\pp(t)))}\in\R,\\
\psi(\qq(t),\pp(t))&=\arg(\qq(t_0))+\tilde{s}(\qq(t),\pp(t)) \alpha(v(\qq(t),\pp(t)))\in\R/2\pi\Z.
 \end{align*}
We finally let $s(\qq(t),\pp(t))$ be the projection of $\tilde{s}(\qq(t),\pp(t))$ to $\R/\Z$.
\end{definition}

\begin{lemma}
If we choose a different point of maximum, then $\tilde{s}(\qq(t),\pp(t))$ changes by an integer and $\psi$ does not change. So the functions $s:\widehat{Y}\to \R/\Z$ and $\psi:\widehat{Y}\to \R/2\pi\Z$ are well-defined.
\end{lemma}

\begin{proof}
Let $t_0<t_1$ be consecutive points of maximum of $|\qq(t)|$. We observe that
\begin{align*}
\tilde{s}(\qq(t),\pp(t))&=\frac{t-t_0}{G(v)}=\frac{t-t_1}{G(v)}+1,\\
\psi(\qq(t),\pp(t))&=\arg(\qq(t_0))+ \frac{t-t_0}{G(v)}\alpha(v)=\arg(\qq(t_0))+\alpha(v)+\frac{t-t_1}{G(v)}\alpha(v)\\
&=\arg(\qq(t_1))+ \frac{t-t_1}{G(v)}\alpha(v)\in\R/2\pi\Z.
\end{align*}
By induction, if we choose a different point of maximum, $\tilde{s}(\qq(t),\pp(t))$ changes by an integer and $\psi(\qq(t),\pp(t))$ does not change. Hence the functions $s$ and $\psi$ are well-defined.
\end{proof}
 \begin{rmk}\label{rem:ag}
It follows from the proof of Lemma \ref{lem:ind} that if $t_0<t_1$ are consecutive critical points of $|\qq(t)|^2$, then $G(v)=2(t_1-t_0)$. Moreover if $v\neq 0$, then $\underline{\alpha}(v)=2(\arg(\qq(t_1))-\arg(\qq(t_0)))$.
\end{rmk}

\begin{rmk}\label{rem:arg}
If $t_0<t_1$ are consecutive critical points of $|\qq(t)|^2$ such that $t_0$ is a point of maximum and $t_1$ is a point of minimum, then
\[\sg(v)\arg(\qq(t_1))=\arg(\qq(t_0))+\frac{\alpha(v)}{2},\]
provided that $v\neq 0$. If $v=0$, then $\alpha(v)=\pi$ and $\arg(\pp(t_1))=\arg(\qq(t_0))+\pi$.
\end{rmk} 
 
 \begin{prop}\label{prop:diffeo}
 The function $(v,s,\psi):\widehat{Y}\to (-M,M)\times \R/\Z\times \R/2\pi\Z$ is a diffeomorphism. Moreover,
\[\lambda= \big(G(v)-\alpha(v) v\big)ds+ v\,d\psi\in\Omega^1(\widehat{Y},\R).\]
 \end{prop}

 \begin{proof}
Step 1: We first show that $(v,s,\psi)$ is smooth.

 First we note that $v$ is smooth. Let $N$ be an open and connected subset of $Y$ which satisfies the following property:
Each Reeb trajectory $\gamma(t)=(\qq(t),\pp(t))$ intersecting $N$ does so in a connected subset and there is exactly one point of maximum $t_0$ and one point of minimum of $t_1>t_0$ of $|\qq(t)|^2$ satisfying $\{\gamma(t_0),\gamma(t_1)\}\subset N$. We observe that $s|_N$ lifts to a function $\tilde{s}:N\to \R$. Moreover, we can define a continuous function $\tilde{\theta}:N\to \R$ such that \[\tilde{\theta}(\qq,\pp)\equiv \arg(\qq(t_0))\quad (\text{mod } 2\pi), \]where $t_0$ is a point of maximum of $|\qq(t)|$ for a Reeb orbit $(\qq(t),\pp(t))$ going through $(\qq,\pp)$ and $\gamma(t_0)\in N$. It follows from Remark \ref{rem:ag} that we can choose $\tilde{s}$ so that $\tilde{s}(\qq(t_j),\pp(t_j))=j/2$ for $j=0,1$.
 These lifts determine a lift $\tilde{\psi}$ of $\psi|_N$ such that $\tilde{\psi}=\tilde{\theta}+\alpha(v)\tilde{s}$. It follows from the smoothness results of differential equations that $\tilde{s}$, $\tilde{\theta}$ and $\tilde{\psi}$ are smooth. Since every point of $\widehat{Y}$ is contained in such a subset $N$, it follows that $(v,s,\psi)$ is smooth.

 Step 2:  We show that the function $(v,s,\psi)$ is a diffeomorphism.

We will construct the inverse function $\Xi:(-M,M)\times\R/\Z\times \R/2\pi\Z \to\widehat{Y}$.
Let $(v_0,s_0,\psi_0)\in(-M,M)\times\R/\Z\times \R/2\pi\Z$. We consider $\tilde{s}_0\in\R$ a pre-image of $s_0$ under the quotient map $\R\to\R/\Z$ and we let \[\theta_0=\psi_0-\alpha(v_0)\tilde{s}_0+\frac{\alpha(v_0)}{2}  \in\R/2\pi\Z.\]  By~\eqref{eq:fder}, $f'(r)>0$ for all $r\in(-\underline{r},\underline{r})$. So $f|_{(-\underline{r},\underline{r})}$ is a diffeomorphism onto $(-M,M)$. Let $r_0=f|_{(-\underline{r},\underline{r})}^{-1}(v_0)$. We now let $\qq_0 =r_0 e^{ i \theta_0}$ and $\pp_0=\sqrt{1-\epsilon U(r_0^2)} ie^{i \theta_0}$. Let $(\qq(t),\pp(t))$ be the Reeb trajectory such that $(\qq(G(v_0)/2),\pp(G(v_0)/2))=(\qq_0,\pp_0)$. We let \[\Xi(v_0,s_0,\psi_0)=(\qq(G(v_0)\tilde{s}_0),\pp(G(v_0)\tilde{s}_0)).\]

We claim that $\Xi$ is well-defined, smooth and that $\Xi$ is the inverse of $(v,s,\psi)$. To see that, we first note that $v(\Xi(v_0,s_0,\psi_0))=r_0 \sqrt{1-\epsilon U(r_0^2)} =f(r_0)=v_0$. It follows from Remark \ref{rem:ag} that the difference in time between two consecutive critical points of $|\qq(t)|^2$ is $G(v_0)/2$. So $0$ is a point of maximum of $|\qq(t)|^2$. Moreover, it follows from \eqref{eq:theta} that $\Xi(v_0,s_0,\psi_0)$ does not depend on the choice of $\tilde{s}_0$. The fact that $\Xi$ is smooth is again a consequence of the smoothness of solutions of differential equations with respect to the initial conditions. From Remark \ref{rem:ag} we also obtain $s(\Xi(v_0,s_0,\psi_0))=\tilde{s}_0=s_0\in \R/\Z$. Now it follows from Remark \ref{rem:arg} that $\arg(\qq(0))=\theta_0-\frac{\alpha(v_0)}{2}$, even in the case when $v_0=0$. Therefore
\begin{equation*}\psi(\Xi(v_0,s_0,\psi_0))
=\arg(\qq(0))+\alpha(v_0)\tilde{s}_0
=\psi_0\in\R/2\Z.\end{equation*}
So $(v,s,\psi)\circ\Xi=\mathbb{I}$.

For the converse, let $(\qq,\pp)\in \widehat{Y}$. Let $(v_0,s_0,\psi_0)=(v,s,\psi)(\qq,\pp)$. So if $0$ is a point of maximum of $|\qq(t)|^2$ on a Reeb trajectory $(\qq(t),\pp(t))$ containing $(\qq,\pp)$, then $(\qq(G(v_0)\tilde{s}_0),\pp(G(v_0)\tilde{s}_0))=(\qq,\pp)$, for some lift $\tilde{s}_0\in\R$ of $s_0\in\R/\Z$. Moreover, $\psi_0=\arg(\qq(0))+\alpha(v_0)\tilde{s}_0$. Now let $r_0:=f|_{(-\underline{r},\underline{r})}^{-1}(v_0)=\sg(v_0) |\qq(G(v_0)/2)|$. From Remarks \ref{rem:ag} and \ref{rem:arg} we obtain \begin{align*}\qq(G(v_0)/2)&=r_0e^{i\left(\arg(\qq(0))+\frac{\alpha(v_0)}{2}\right)}=r_0 e^{i\left(\psi_0-\alpha(v_0)\tilde{s}_0+\frac{\alpha(v_0)}{2}\right)},\\
\pp(G(v_0)/2)&=\sqrt{1-\epsilon U(r_0^2)}i e^{i\left(\arg(\qq(0))+\frac{\alpha(v_0)}{2}\right)}=\sqrt{1-\epsilon U(r_0^2)}i e^{i\left(\psi_0-\alpha(v_0)\tilde{s}_0+\frac{\alpha(v_0)}{2}\right)}.\end{align*}
It follows from the uniqueness of solutions of differential equations that $(\qq(t),\pp(t))$ is the trajectory used to define $\Xi$ above. Therefore
\[\Xi(v_0,s_0,\psi_0)=(\qq(G(v_0)\tilde{s}_0),\pp(G(v_0)\tilde{s}_0))=(\qq,\pp).\]

Step 3: We will show that $\lambda-v\,d\tilde{\theta}=d(G(v)\tilde{s})$ on $N$ for every $N$ as defined in Step 1.

We first observe that it follows from Step 2 that $(v,\tilde{s},\tilde{\psi}):N\to (-M,M)\times\R^2$ is a smooth chart, where $\tilde{s}$ and $\tilde{\psi}$ are the lifts of $s$ and $\psi$ which are defined in Step 1.
 Since $\tilde{\psi}=\tilde{\theta}+\alpha(v)\tilde{s}$,
 \begin{equation}\label{eq:diffpsi}
 d\tilde{\psi}=d\tilde{\theta}+\tilde{s}\alpha'(v)dv +\alpha d\tilde{s}\in\Omega^1(N).
 \end{equation}

Now let $\eta=\lambda-v\,d\tilde{\theta}-d(G(v)\tilde{s})$. We will show that $\eta=0$ on $N$. We observe that $d\tilde{\theta}(R)=0$ and that $d(G(v) \tilde{s})(R)=1$ whence $\eta(R)=0$. Moreover $dv(R)=0$ so $\mathcal{L}_R\eta =d(\eta(R))+(d\lambda-dv\wedge d\tilde{\theta})(R,\cdot)=0$. Hence the Reeb flow preserves $\eta$.

Let $D\subset N$ be the set of all points $(\qq,\pp)\in N$ such that $|\qq|$ is the maximum of $|\qq(t)|$ for a Reeb trajectory $(\qq(t),\pp(t))$ through $(\qq,\pp)$. We now claim that $\eta=0$ on $D$. Following the usual notation, we will denote by $\partial/\partial v$ the vector field on $N$ satisfying \[dv(\partial/\partial v)=1\quad\text{and}\quad d\tilde{s}(\partial/\partial v)=d\tilde{\psi}(\partial/\partial v)=0.\]
Let $W$ be the vector field on $N$ defined by \[W=p_2\frac{\partial}{\partial q_1}-p_1\frac{\partial}{\partial q_2}+\epsilon U'(|\qq|^2)\Big(q_2\frac{\partial}{\partial p_1}-q_1\frac{\partial}{\partial p_2}\Big).\]
At a point $(\qq,\pp)\in D$, the vectors $\qq$ and $\pp$ are perpendicular, so the flow of $W$ preserves $\arg(\pp)$ and $\arg(\qq)$ and does not change $\tilde{s}$ and $\tilde{\psi}$. So $d\tilde{s}(W)|_D=d\tilde{\psi}(W)|_D=0$ and hence $\partial/\partial v$ is parallel to $W$ along $D$. So $\lambda(\partial/\partial v)|_D=\lambda(W)|_D=0$. Moreover since $\tilde{s}|_D=0$, it follows from~\eqref{eq:diffpsi} that $d\tilde{\theta}(\partial/\partial v)|_D=0$. We now let $X$ be the vector field on $N$ defined by
\[X=-p_2\frac{\partial}{\partial p_1}+p_1\frac{\partial}{\partial p_2}-q_2\frac{\partial}{\partial q_1}+q_1\frac{\partial}{\partial q_2}.\]
The flow of $X$ is the exponential of the $SO(2)$-action. Hence $d\tilde{\theta}(X)=1$ and $dv(X)=0$. 
Since the flow of $X$ preserves $D$, we have $d\tilde{s}(X)|_D=0$. So
\begin{align*}
\eta(R)|_D&=0,\\
\eta\bigg(\frac{\partial}{\partial v}\bigg)\Bigg|_D&=\lambda\bigg(\frac{\partial}{\partial v}\bigg)\Bigg|_D-\left(v\, d\tilde{\theta}\bigg(\frac{\partial}{\partial v}\bigg)+G(v) d\tilde{s}\bigg(\frac{\partial}{\partial v}\bigg)\right)\Bigg|_D=0,\\
\eta(X)|_D&=(\lambda(X)-v\, d\tilde{\theta}(X))|_D-G(v) d\tilde{s}(X)|_D=(v-v)|_D-0=0.
\end{align*}
Since $W$ and $X$ are linearly indepedent and $W|_D$ and $X|_D$ belong to $\ker(\lambda)|_D$, it follows that $R|_D$, $W|_D$ and $X|_D$ are linearly independent.
Therefore $\eta|_D=0$.
Since the Reeb flow preserves $\eta$, it follows that $\eta=0$ on $N$.

Step 4: We conclude the proof of the proposition.

Let $\bar{D}\subset N$ be the set of all points $(\qq,\pp)\in N$ such that $|\qq|^2$ is the minimum of $|\qq(t)|^2$ for a Reeb trajectory $(\qq(t),\pp(t))$ through $(\qq,\pp)$. Then $\tilde{s}(\qq,\pp)=1/2$, for every $(\qq,\pp)\in \bar{D}$. By~\eqref{eq:diffpsi}, $d\tilde{\theta}(\partial/\partial v)=-\alpha'(v)/2$ at all points on $\bar{D}$. As in Step 3, at $(\qq,\pp)\in \bar{D}$, the vector $W$ is parallel to $\partial/\partial v$. So $\lambda(\partial/\partial v)|_{\bar{D}}=0$. It follows from Step 3 that
\[\frac{1}{2}\alpha'(v)v=\big(\lambda-v\, d \tilde{\theta}\big)\bigg(\frac{\partial}{\partial v}\bigg)\Bigg|_{\bar{D}}=\big(G(v) d\tilde{s}+\tilde{s}G'(v)\,dv\big)\bigg(\frac{\partial}{\partial v}\bigg)\Bigg|_{\bar{D}}=\frac{1}{2}G'(v).\]
So $G'(v)=\alpha'(v)v$. It follows from~\eqref{eq:diffpsi} that
\[\lambda= v\,d\tilde{\theta}+G(v) d\tilde{s}+\tilde{s}G'(v)\,dv=
\big(G(v)-\alpha(v)v)d\tilde{s}+ v\,d\tilde{\psi}.\]
Since this equation holds for all $N$, and since $d\tilde{s}=ds$ and $d\tilde{\psi}=d\psi$,
\begin{equation*}\label{eq:lambda}
\lambda=\big(G(v)-\alpha(v) v\big)ds+ v\,d\psi.
\end{equation*}
 \end{proof}

 We can now define the toric coordinates $(\phi_1,\phi_2)$ as follows.
\begin{definition}\label{def:coord}For $\xx\in\widehat{Y}$ we let
 \begin{align*}
\phi_1(\xx)&= s(\xx)-\frac{\psi(\xx)}{2\pi}\in\R/\Z,\\
\phi_2(\xx)&= \frac{\psi(\xx)}{2\pi}\in \R/\Z.
\end{align*}
\end{definition}
The following corollary is a straight-forward consequence of Proposition \ref{prop:diffeo}.
\begin{cor}\label{cor:diffeo}
The function $(v,\phi_1,\phi_2):\widehat{Y}\to(-M,M)\times (\R/\Z)^2$ is a diffeomorphism. Moreover,
\[\lambda= \big(G(v)-\alpha(v)v\big)d\phi_1+\big(G(v)+(2\pi-\alpha(v)) v\big)d\phi_2.\]
\end{cor}

\subsection{The extension to $Y$}\label{sub:ext}

We now define $\widehat{\Phi}:\widehat{Y}\to\C^2$ by \begin{equation}\label{eq:defpsi}\widehat{\Phi}(\xx)=\left(\sqrt{\frac{\rho_1(\xx)}{\pi}} \,e^{2\pi i\phi_1(\xx)},\sqrt{\frac{\rho_2(\xx)}{\pi}}\, e^{2\pi i\phi_2(\xx)}\right),\end{equation}
where $\rho_1(\xx)=(G(v)-\alpha(v)v)(\xx)$ and $\rho_2(\xx)=(G(v)+(2\pi-\alpha(v))v)(\xx)$. In the following technical lemma, we will show that $\widehat{\Phi}$ is well-defined and that it extends to an embedding of $Y$ into $\C^2$. 
\begin{lemma}\label{lem:Phi}
 The functions $\rho_1$ and $\rho_2$ are positive functions and they extend to $Y$ such that $\rho_1|_{C_+}=0$ and $\rho_2|_{C_-}=0$. Moreover $\widehat{\Phi}$ can be extended to a smooth embedding $\Phi:Y\to\C^2$ satisfying $\Phi^*\lambda=\lambda$.
  \end{lemma}

\begin{proof}
Step 1: We first assume that $\widehat{\Phi}$ can be extended to a smooth embedding $\Phi$. We claim that $\Phi^*\lambda=\lambda$.

We can write
\[\lambda=\frac{1}{2}\sum_{i=1}^2 r_i^2 d\theta_i,\]
where $z_j=r_je^{i\theta_j}$ are the coordinates of $\C^2$. It follows from \eqref{eq:defpsi} and from Corollary \ref{cor:diffeo} that
\[\widehat{\Phi}^*\lambda=\frac{1}{2}\sum_{i=1}^2 \frac{\rho_i}{\pi} d(2\pi \phi_i)=\sum_{i=1}^2 \rho_i d\phi_i=\lambda. \]
By continuity, we conclude that $\Phi^*\lambda=\lambda$ on all of $Y$.

Step 2: We now reduce the rest of the proof to showing that $\rho_1(\xx)$ and $\rho_2(\xx)$ are positive if $v(\xx)\ge 0$, that $\rho_1|_{C_+}=0$ and that $\widehat{\Phi}$ smoothly extends to $C_+$.

Let $\xx=(\qq,\pp)\in \widehat{Y}$ and let $\tilde{\xx}=(\qq,-\pp)$. So $v(\tilde{\xx})=-v(\xx)$. We claim that $G\circ v (\tilde{\xx})=G\circ v(\xx)$ and $\alpha\circ v (\tilde{\xx})=2\pi-\alpha\circ v(\xx)$. Indeed if $(\qq(t),\pp(t))$ is a Reeb trajectory going through $\xx$, then $(\qq(-t),-\pp(-t))$ is a Reeb trajectory going through $\tilde{\xx}$. If $t_0<t_1$ are consecutive points of maximum of $|\qq(t)|$ then $-t_1<-t_0$ are two consecutive points of maximum of $|\qq(-t)|$. So \[G\circ v(\xx)=t_1-t_0=(-t_0)-(-t_1)=G\circ v(\tilde{\xx}).\]
Moreover $\arg(\qq(t_1))-\arg(\qq(t_0))=-(\arg(-\qq(-t_0))-\arg(-\qq(-t_1)))\in\R/\Z$. This implies that $\alpha\circ v(\xx)\equiv-\alpha\circ v(\tilde{\xx})\;(\text{mod } 2\pi)$. Since $\alpha(0)=\pi$, it follows that $\alpha\circ v(\tilde{\xx})=2\pi-\alpha\circ v(\xx)$. 
Hence we obtain:
\begin{equation}
\begin{aligned}
\rho_1(\tilde{\xx})&=G\circ v(\xx)-\left(2\pi -\alpha\circ v(\xx)\right)(-v(\xx))\\
&=G\circ v(\xx)+(2\pi-\alpha\circ v(\xx))v(\xx)=\rho_2(\xx).
\end{aligned}\label{eq:mv}
\end{equation}
It follows from~\eqref{eq:mv} that if $\rho_1(\xx)$ and $\rho_2(\xx)$ are positive for $v(\xx)>0$, then they are also positive for $v(\xx)<0$. Moreover if $\rho_1|_{C_+}=0$, then $\rho_2|_{C_-}=0$.

By a similar reasoning as above, we can deduce that:
\begin{equation*}
s(\tilde{\xx})=s(\xx),\qquad \psi(\tilde{\xx})=2\pi s(\xx)-\psi(\xx).
\end{equation*}
Hence
\begin{equation}\label{eq:phi12}
\phi_1(\tilde{\xx})=\phi_2(\xx).
\end{equation}
From \eqref{eq:mv} and \eqref{eq:phi12} we conclude that $\widehat{\Phi}$ can be smoothly extended to $C_+$ if, and only if, it can be smoothly extended to $C_-$.

Step 3: We will express the functions $s$, $\psi$, $G$ and $\alpha$ as integrals in preparation for the next steps.

Let $\xx\in \widehat{Y}$ such that $v(\xx)>0$ and let $(\qq(t),\pp(t))$ be a Reeb trajectory such that $\xx=(\qq(t_{\xx}),\pp(t_{\xx}))$, for some $t_\xx$. We also let $t_0$ the largest point of maximum of $|\qq(t)|^2$ with $t_0< t_{\xx}$.

Let $t_1$ be the following critical point of $|\qq(t)|^2$, i.e., $t_1>t_0$ and there is no critical point between $t_0$ and $t_1$. Then $t_1$ is a point of minimum of $|\qq(t)|^2$. Now let $r(t)$ and $\theta(t)$ be the polar coordinates of $\qq(t)$ and let $r_i=r(t_i)$, for $i=0,1$. 
Recall the function $f(r)=r\sqrt{1-\epsilon U(r^2)}$. So $v(\xx)=f(r_0)=f(r_1)$. From \eqref{eq:K2} it follows that \[v(\xx)=\qq(t)\times \pp(t)=K(r(t)^2) \qq(t)\times\dot{\qq}(t)= K(r(t)^2) r(t)^2\dot{\theta}(t),\] where $K(u)=\frac{1}{2}(1-\epsilon U(u)+\epsilon U'(u)u)$. So
\begin{equation}\label{eq:dottheta}\dot{\theta}=\frac{v(\xx)}{K(r^2)r^2}=\frac{f(r_0)}{K(r^2)r^2}.\end{equation}
From the first equation of~\eqref{eq:polar} and~\eqref{eq:dottheta}, we obtain:
\begin{equation}\label{eq:rddot1}
K(r^2)^2\ddot{r}+2K(r^2)K'(r^2) r (\dot{r})^2=\frac{f(r_0)^2}{r^3}-\epsilon U' (r^2)r.\end{equation}

We will first assume that $t_0<t_\xx\le t_1$. This is equivalent to requiring that $s(\xx)\in(0,1/2]$. So $\dot{r}<0$ in $(t_0,t_1)$, which implies that $r$ is invertible in $(t_0,t_1)$. So we can write $\dot{r}$ as a function of $r$ in $(t_0,t_1)$. Integrating~\eqref{eq:rddot1} in $(r_0,r_1)$ and using $\dot{r}(t_0)=0$, we obtain:
\begin{equation*}(K(r^2)\dot{r})^2=\frac{f(r)^2-f(r_0)^2}{r^2}.\end{equation*} Therefore
\begin{equation}\label{eq:rdot}\dot{r}=-\frac{\sqrt{f(r)^2-f(r_0)^2}}{K(r^2)r}.\end{equation}
If we write $u=r^2$, then \eqref{eq:rdot} is equivalent to:
\begin{equation}\label{udot}
\dot{u}=-2\frac{\sqrt{F(u)-F(u_0)}}{K(u)},
\end{equation}
where $F(u)=u(1-\epsilon U(u))$ and $u_0=r_0^2$. Let $\bar{u}$ be the (unique) positive critical point of $F$. So \begin{equation*}\label{eq:epsu}1-\epsilon U(\bar{u})-\epsilon U'(\bar{u})\bar{u}=0\end{equation*} and $\bar{u}=\bar{r}^2$, where $\bar{r}>0$ is the radius of the circles $C_+$ and $C_-$. Note also that $F(\bar{u})=M^2=f(\bar{r})^2$ and that $v(\xx)=f(r_0)=\sqrt{F(u_0)}$. Let $u_\xx=u(t_\xx)$.
Since $\dot{u}<0$ in $(t_0,t_1)$, it follows that $u$ is invertible in $(t_0,t_1)$.

 It follows from  \eqref{eq:dottheta} and \eqref{udot} that:
\begin{align}
\label{eq:s2}
&\begin{aligned}
s(\xx)&=\frac{t_\xx-t_0}{G\circ v(\xx)}=\frac{1}{G\circ v(\xx)}\int_{u_0}^{u_\xx} t'(u)\,du\\
&=\frac{1}{2G\circ v(\xx)}\int_{u_\xx}^{u_0}\frac{K(u)}{\sqrt{F(u)-F(u_0)}}\,du,\end{aligned}
\\
&\label{eq:psi2}
\begin{aligned}
\psi(\xx)&=\arg(\qq(t_0))+\frac{\alpha\circ v(\xx)}{G\circ v(\xx)}\cdot (t_\xx-t_0)\\&=\arg(\qq(t_\xx))-\int_{t_0}^{t_\xx} \dot{\theta}(t)\,dt+\frac{\alpha\circ v(\xx)}{G\circ v(\xx)}\int_{u_0}^{u_\xx} t'(u)\,du\\
&=\arg(\qq(t_\xx))-\int_{u_\xx}^{u_0}\frac{v(\xx)}{2u\sqrt{F(u)-F(u_0)}}\,du+\frac{\alpha\circ v(\xx)}{2G\circ v(\xx)}\int_{u_\xx}^{u_0}\frac{K(u)}{\sqrt{F(u)-F(u_0)}}\,du.
\end{aligned}
\end{align}

Since $F'(\bar{u})=0$, we can write
\begin{equation*}
F(u)=F(\bar{u})-H(u)(u-\bar{u})^2,
\end{equation*}
where $H:(u_1,u_0)\to \R$ is a smooth function. Moreover, $H>0$ since $F(\bar{u})$ is the maximum of $F$ and $H(\bar{u})=-\frac{1}{2}F''(\bar{u})>0$. 
For $u\in(u_1,u_0)$, let $W(u):=(H(u))^{1/2}(u-\bar{u})$ and let $A(\xx)=\sqrt{F(\bar{u})-F(u_0)}=W(u_0)=-W(u_1)$. We note that $W'(u)=-F'(u)/(2W(u))>0$ for $u\neq \bar{u}$ and $W'(\bar{u})=\sqrt{H(\bar{u})}=\sqrt{-\frac{1}{2}F''(\bar{u})}>0$. So $W$ is a diffeomorphism onto $(-A(\xx),A(\xx))$. We change variables in \eqref{eq:s2} and \eqref{eq:psi2} by letting $\zeta=\arcsin\left(\frac{1}{A(\xx)}W(u)\right)$. Let $\zeta_\xx=\arcsin\left(\frac{1}{A(\xx)}W(u_\xx)\right)$.

Therefore we conclude that if $s(\xx)\in(0,1/2]\subset\R/\Z$, then
\begin{align}
s(\xx)&=\frac{1}{2G(v(\xx))}\int_{\zeta_{\xx}}^{\pi/2}\frac{K\circ W^{-1}}{W'\circ W^{-1}}(A(\xx)\sin\zeta)\,d\zeta\label{eq:si},\\
\psi(\xx)&=\arg(\qq(t_\xx))-\frac{v(\xx)}{2} \int_{\zeta_\xx}^{\pi/2}\frac{1}{W^{-1}\cdot (W'\circ W^{-1})}(A(\xx)\sin\zeta)\,d\zeta\nonumber\\&+\frac{\alpha\circ v(\xx)}{2G\circ v(\xx)}\int_{\zeta_\xx}^{\pi/2}\frac{K\circ W^{-1}}{W'\circ W^{-1}}(A(\xx)\sin\zeta)\,d\zeta\label{eq:psii}.
\end{align}

Now we assume that $t_1<t_\xx\le 2t_1-t_0$ which implies that $s(\xx)\in(1/2,1]\subset\R/\Z$. By an analogous calculation and using Remark \ref{rem:ag}, we obtain
\begin{align}
s(\xx)&=\frac{t_\xx-t_0}{G\circ v(\xx)}=1-\frac{(2t_1-t_0)-t_\xx}{G\circ v(\xx)}\nonumber\\
&=1-\frac{1}{G\circ v(\xx)}\int_{u_0}^{u_\xx} t'(u)\,du\nonumber\\
&=-\frac{1}{2G\circ v(\xx)}\int_{\zeta_{\xx}}^{\pi/2}\frac{K\circ W^{-1}}{W'\circ W^{-1}}(A(\xx)\sin\zeta)\,d\zeta.
\label{eq:si2}
\end{align}
We note that the last equality holds since $s(\xx)\in\R/\Z$.

Similarly, if $s(\xx)\in(1/2,1]$, then:
\begin{align}
\psi(\xx)&=\arg(\qq(t_\xx))-\int_{t_1}^{t_\xx} \dot{\theta}(t)\,dt+\alpha\circ v(\xx) \frac{1-(2t_1-t_0-t_\xx)}{G\circ v(\xx)}
\nonumber\\
&=
\arg(\qq(t_\xx))-\left(-\int_{t_0}^{t_\xx} \dot{\theta}(t)\,dt+\frac{\alpha\circ v(\xx)}{G\circ v(\xx)}\int_{u_0}^{u_\xx} t'(u)\,du\right)\nonumber\\
&=\arg(\qq(t_\xx))-\bigg(-\frac{v(\xx)}{2} \int_{\zeta_\xx}^{\pi/2}\frac{1}{W^{-1}\cdot (W'\circ W^{-1})}(A(\xx)\sin\zeta)\,d\zeta\nonumber\\
&\qquad\qquad\qquad\qquad+\frac{\alpha\circ v(\xx)}{2G\circ v(\xx)}\int_{\zeta_\xx}^{\pi/2}\frac{K\circ W^{-1}}{W'\circ W^{-1}}(A(\xx)\sin\zeta)\,d\zeta\bigg).\label{eq:psii2}
\end{align}

We will now write integral formulas for $\alpha\circ v$ and $G\circ v$. Since $v(\xx)>0$, it follows from Remark \ref{rem:ag} that $\alpha\circ v(\xx)=2\int_{t_0}^{t_1}\dot{\theta}(t)dt$. So, by a similar calculation to the one above, we obtain:
\begin{align}
\alpha\circ v(\xx)&=v(\xx)\int_{-\pi/2}^{\pi/2}\frac{1}{W^{-1}\cdot (W'\circ W^{-1})}(A(\xx)\sin\zeta)\,d\zeta,\label{eq:alpha}\\
G\circ v(\xx)&=\int_{-\pi/2}^{\pi/2}\frac{K\circ W^{-1}}{W'\circ W^{-1}}(A(\xx)\sin\zeta)\,d\zeta.\label{eq:G}
\end{align}

Step 4: We now show that $\rho_1(\xx)>0$ and $\rho_2(\xx)>0$ for all $\xx$ satisfying $v(\xx)\ge 0$.

Let $\xx\in Y$ and let $v=v(\xx)$. Assume that $v(\xx)>0$. Let $\zeta\in(-\pi/2,\pi/2)$ and $u=W^{-1}(A(\xx)\sin\zeta)$. It follows directly from the definition of $K$ and $F$ that $K(u)=(1-\epsilon U(u))-F'(u)/2$. So
\begin{align*}
\frac{K(u)}{W'(u)}&=\left(\frac{F(u)}{u}-\frac{1}{2}F'(u)\right)\cdot\frac{1}{W'(u)}=\frac{M^2-W(u)^2}{uW'(u)}-\frac{F'(u)}{2W'(u)} .\end{align*}
Since $W(u)^2=M^2-F(u)$, it follows that $2W(u)W'(u)=-F'(u)$ and hence
\begin{equation}\label{eq:K}\frac{K(u)}{W'(u)}= \frac{M^2-W(u)^2}{u\cdot W'(u)}+W(u).\end{equation}
So
\begin{equation}\label{eq:diff1}
\begin{aligned}
\frac{K(u)}{W'(u)}-\frac{v}{uW'(u)}\cdot v&=\frac{M^2-W(u)^2-v^2}{u\cdot W'(u)}+W(u)\\&=\frac{A(\xx)^2\cos^2\zeta}{(W^{-1}\cdot (W'\circ W^{-1}))(A(\xx)\sin\zeta)}+A(\xx)\sin\zeta.
\end{aligned}
\end{equation}
Therefore from \eqref{eq:alpha}, \eqref{eq:G} and \eqref{eq:diff1}, we obtain:
\begin{align}
\rho_1(\xx)=G(v)-\alpha(v)v&=\int_{-\pi/2}^{\pi/2}\left( \frac{A(\xx)^2\cos^2\zeta}{(W^{-1}\cdot (W'\circ W^{-1}))(A(\xx)\sin\zeta)}+A(\xx)\sin\zeta\right)\,d\zeta\nonumber\\
&=\int_{-\pi/2}^{\pi/2}\frac{A(\xx)^2\cos^2\zeta}{(W^{-1}\cdot (W'\circ W^{-1}))(A(\xx)\sin\zeta)} \,d\zeta
.\label{eq:sumin}
\end{align}
The integral in \eqref{eq:sumin} is strictly positive, since the integrand is strictly positive in $(-\pi/2,\pi/2)$. Therefore $\rho_1(\xx)>0$ if $v(\xx)>0$. Moreover, if $v(\xx)=0$, then $A(\xx)^2=M^2$. By continuity, it follows from \eqref{eq:sumin} that
\begin{equation*}\label{eq:rho1}\rho_1(\xx)=\int_{-\pi/2}^{\pi/2} \frac{M^2\cos^2\zeta}{(W^{-1}\cdot (W'\circ W^{-1}))(M\sin\zeta)}\,d\zeta>0.\end{equation*}
Finally, since we are assuming that $v(\xx)\ge 0$, we obtain:
\[\rho_2(\xx)=\rho_1(\xx)+2\pi v(\xx)\ge \rho_1(\xx)>0.\]

Step 5: We will prove that $G\circ v$ and $\alpha\circ v$ can be smoothly extended to $C_+$ and that $\rho_1|_{C_+}=0$.

Let $n>1$ be a natural number. Since $\frac{K\circ W^{-1}}{W'\circ W^{-1}}$ is smooth, it follows that there exist $c_1,\dots,c_{2n+1}\in\R$, such that
\begin{equation*}
\frac{K\circ W^{-1}}{W'\circ W^{-1}}(y)=\frac{K(\bar{u})}{W'(\bar{u})}+\sum_{j=1}^{2n+1} c_j y^j+O(y^{2n+2}),
\end{equation*}
where $O$ is a continuous function satisfying $|O(y)^{2n+2}|\le cy^{2n+2}$, for some constant $c>0$ and for $y$ sufficiently small.
So
\begin{align}
G\circ v(\xx)&=\frac{\pi K(\bar{u})}{W'(\bar{u})}+\sum_{j=1}^{2n+1}\int_{-\pi/2}^{\pi/2} c_j A(\xx)^j\sin^j \zeta\,d\zeta+\int_{-\pi/2}^{\pi/2} O((A(\xx)\sin\zeta)^{2n+2})\,d\zeta\nonumber\\
&=\frac{\pi K(\bar{u})}{W'(\bar{u})}+\sum_{j=1}^{n}\int_{-\pi/2}^{\pi/2} c_{2j} A(\xx)^{2j}\sin^{2j} \zeta\,d\zeta+O(A(\xx)^{2n+2})\label{eq:sumGv}.
\end{align}
Since $A(\xx)^2=M^2-v(\xx)^2$, it follows that $A^2$ can be smoothly extended to $C_+$. Let $\xx_+\in C_+$ and let $\mathcal{N}$ be a neighborhood of $\xx_+$. So there exists a constant $c>0$ such that $A(\xx)^2\le c|\xx-\xx_+|$ for all $\xx\in \mathcal{N}$. Hence $O(A(\xx)^{2n+2})\le O\left(|\xx-\xx_+|^{n+1}\right).$
Therefore the sum \eqref{eq:sumGv} defines a function on $\mathcal{N}$ which is $n$ times differentiable. Since $n$ and $\xx_+$ were arbitrary, it follows that $G\circ v$ can be smoothly extended to $C_+$ so that $G\circ v|_{C_+}=\frac{\pi K(\bar{u})}{W'(\bar{u})}$.

Analogously, $\alpha\circ v$ can be smoothly extended to $C_+$ so that $\alpha\circ v|_{C_+}=\frac{\pi M}{\bar{u}\cdot W'(\bar{u})}$. Moreover, using \eqref{eq:sumin} and proceeding as above, we conclude that there exists a smooth real function $\tilde{h}$ satisfying $\tilde{h}(0)=0$ such that
\begin{align*}\rho_1(\xx)&=A(\xx)^2\left(\frac{M^2}{\bar{u}W'(\bar{u})}\int_{-\pi/2}^{\pi/2} \cos^2\zeta\,d\zeta+\tilde{h}(A(\xx)^2)\right)\\&=A(\xx)^2\left(\frac{M^2\pi}{2\bar{u}W'(\bar{u})}+\tilde{h}(A(\xx)^2)\right).\end{align*}
So there exists a smooth function $h:Y \to \R$, such that
\begin{equation}\label{eq:h}
\sqrt{\rho_1(\xx)}=A(\xx) h(\xx).
\end{equation}
In particular, $\rho_1|_{C_+}=0$.

Step 6: We will now define smooth functions $\xi_1$ and $\xi_2$ in a neighborhood of $C_+$, such that $A(\xx)^2=\xi_1(\xx)^2+\xi_2(\xx)^2$.

We first define a function $\Upsilon$ as follows. 
\[\begin{array}{cccl}
\Upsilon:&C_+\times D^2&\to& Y\\ 
&\left((\qq,\pp),y,z\right)&\mapsto&\left((1+y) \qq,\displaystyle \sqrt{\frac{1-\epsilon U((1+y)^2\bar{u})}{1-\epsilon U(\bar{u})}} e^{iz} \pp\right).
\end{array}\]
Here $D^2$ denotes a disk with a small radius. We observe that $\Upsilon$ is a diffeomorphism onto a neighborhood $\mathcal{N}$ of $C_+$ provided that the radius of the disk is small enough.
For $\xx\in\mathcal{N}$, we let $\xi_1(\xx)=W(u_\xx)=W((1+y)^2\bar{u})$. So $\xi_1$ is smooth.

Now we consider the function $L(\xx)=A(\xx)^2-\xi_1(\xx)^2$. By a simple calculation we obtain:\begin{equation*}
L\circ\Upsilon((\qq,\pp),y,z)=F((1+y)^2\bar{u})-v(\Upsilon((\qq,\pp),y,z)))^2=F((1+y)^2\bar{u})\sin^2 z.
\end{equation*}
Since $|z|$ is small, the set of critical points is $\{z=0\}\cap (C_+\times D^2)$ which is a smooth submanifold of $\mathcal{N}$. Moreover
\begin{equation*}
\frac{\partial^2(L\circ\Upsilon)}{\partial z^2}((\qq,\pp),y,0)=2F((1+y)^2\bar{u})>0.
\end{equation*}
So $L$ is a Morse-Bott function. Therefore there exists a smooth function $\xi_2$ defined in a neighborhood of the critical set (which we could assume to be $\mathcal{N}$ by possibly shrinking it) so that $L(\xx)=\xi_2(\xx)^2$. Therefore
\begin{equation*}
M^2-v(\xx)^2=A(\xx)^2=\xi_1(\xx)^2+\xi_2(\xx)^2.
\end{equation*}
We observe that we can choose the function $\xi_2$ so that \begin{equation}\label{eq:sxi}\xi_2(\xx)> 0\iff s(\xx)\in(0,1/2)\quad\text{and}\quad\xi_2(\xx)<0\iff s(\xx)\in(1/2,1).\end{equation}

Step 7: We now show that $\psi$ can be smoothly extended to $C_+$.

We first let $J(u)=(W^{-1}\cdot (W'\circ W^{-1}))(u)$. It follows from \eqref{eq:psii}, \eqref{eq:psii2}, \eqref{eq:alpha}, \eqref{eq:G}, \eqref{eq:K} and \eqref{eq:sxi} that, for $\xx\in\mathcal{N}\setminus C_+$,
\begin{equation}\label{eq:psiarg}
\psi(\xx)-\arg(\q(t_\xx))=\frac{v(\xx)}{2G\circ v(\xx)}\cdot (D(\xx)+\alpha\circ v(\xx)\cdot \xi_2(\xx)),
\end{equation}
where
\begin{equation}\label{eq:d1}
\begin{aligned}
D(\xx)=\sg(\xi_2(\xx))&\left(-\int_{-\pi/2}^{\pi/2}\int_{\zeta_x}^{\pi/2}\frac{M^2-A(\xx)^2\sin^2\zeta}{J(A(\xx)\sin\zeta)J(A(\xx)\sin\bar{\zeta})}\, d\bar{\zeta}\,d\zeta\right.\\&\quad\left.+\int_{\zeta_x}^{\pi/2}\int_{-\pi/2}^{\pi/2}\frac{M^2-A(\xx)^2\sin^2\zeta}{J(A(\xx)\sin\zeta)J(A(\xx)\sin\bar{\zeta})}\,d\bar{\zeta}\, d\zeta\right).\end{aligned}
\end{equation}
We observe that $\lim_{x\to C_+} D(\xx)=0$, so we can extend $\psi$ and $D$ to $C_+$, by letting $D|_{C_+}=0$. In particular, \eqref{eq:psiarg} and \eqref{eq:d1} hold for all $\xx\in\mathcal{N}$.
Since the function $\xx\mapsto\arg(\q(t_\xx))$ is smooth in $\mathcal{N}$, it follows from Step 5 and \eqref{eq:psiarg} that it suffices to show that $D$ is smooth.

From \eqref{eq:d1} we obtain:
\begin{align}
\sg(\xi_2(\xx))D(\xx)=&-\int_{-\pi/2}^{\zeta_x}\int_{\zeta_x}^{\pi/2}\frac{M^2-A(\xx)^2\sin^2\zeta}{J(A(\xx)\sin\zeta)J(A(\xx)\sin\bar{\zeta})}\, d\bar{\zeta}\,d\zeta\nonumber\\&+\int_{\zeta_x}^{\pi/2}\int_{-\pi/2}^{\zeta_x}\frac{M^2-A(\xx)^2\sin^2\zeta}{J(A(\xx)\sin\zeta)J(A(\xx)\sin\bar{\zeta})}\, d\bar{\zeta}\,d\zeta\nonumber\\=&\int_{-\pi/2}^{\zeta_x}\int_{\zeta_x}^{\pi/2}\frac{A(\xx)^2\sin^2\zeta}{J(A(\xx)\sin\zeta)J(A(\xx)\sin\bar{\zeta})}\, d\bar{\zeta}\,d\zeta\nonumber\\&-\int_{\zeta_x}^{\pi/2}\int_{-\pi/2}^{\zeta_x}\frac{A(\xx)^2\sin^2\zeta}{J(A(\xx)\sin\zeta)J(A(\xx)\sin\bar{\zeta})}\, d\bar{\zeta}\,d\zeta
\nonumber\\=&A(\xx)^2\int_{-\pi/2}^{\zeta_x}\int_{\zeta_x}^{\pi/2}\frac{\sin^2\zeta-\sin^2\bar{\zeta}}{J(A(\xx)\sin\zeta)J(A(\xx)\sin\bar{\zeta})}\, d\bar{\zeta}\,d\zeta.\label{eq:d2}
\end{align}
Since $J$ is a smooth function bounded away from $0$, we can write
\begin{equation}
\frac{1}{J(y)}=\sum_{j=0}^{2n+1} c_j y^n +O(y).\label{eq:isum}
\end{equation}
It follows from \eqref{eq:d2} and \eqref{eq:isum} that
\begin{align*}
D(\xx)&=\sg(\xi_2(\xx))\int_{-\pi/2}^{\zeta_x}\int_{\zeta_x}^{\pi/2}\sum_{j,k=0}^{2n+1}c_jc_{k}A(\xx)^{j+k+2}\sg(\xi_2(\xx))\sin^j\zeta \sin^{k}\bar{\zeta}(\sin^2\zeta-\sin^2\bar{\zeta})\, d\bar{\zeta}\,d\zeta\\&\qquad+O(A(\xx)^{2n+2})\\
&=\frac{1}{2}\sum_{j,k=0}^{2n+1}c_jc_kA(\xx)^{j+k}I_{jk}(\xx)+O(A(\xx)^{2n+2}),
\end{align*}
where
\begin{equation}\label{eq:ijk}
I_{jk}(\xx)=\int_{-\pi/2}^{\zeta_x}\int_{\zeta_x}^{\pi/2}(\sin^j\zeta\sin^k\bar{\zeta}+\sin^k\zeta\sin^j\bar{\zeta})(\sin^2\zeta-\sin^2\bar{\zeta})d\bar{\zeta}d\zeta.
\end{equation}

We now claim that $A(\xx)^{j+k+2}\sg(\xi_2(\xx))I_{jk}(\xx)$ is smooth. It follows from a simple calculation that the function
\[I^j(\zeta):=(\cos\zeta) \cdot P_j(\sin\zeta)+C_j \cdot \zeta\] is a primitive of the function $\sin^j \zeta$,
where $P_j$ is a degree $j-1$ polynomial such that the degree of all terms of $P_j$ has the opposite parity as that of $j$, and $C_j$ is a non-negative constant which equals $0$ if $j$ is odd. We recall that $\sin\zeta_{\xx}=\frac{\xi_1(\xx)}{A(\xx)}$ and $\cos\zeta_{\xx}=\frac{|\xi_2(\xx)|}{A(\xx)}$. By another simple calculation, we obtain
\begin{align}
I^j(\zeta)\bigg|_{\zeta=\zeta_{\xx}}^{\zeta=\pi/2}&=C_j\cdot\left(\frac{\pi}{2}-\zeta_{\xx}\right)-\frac{|\xi_2(\xx)|\cdot\tilde{P}_j(\xi_1(\xx),A^2(\xx))}{A(\xx)^j},\label{eq:ij1}\\
I^j(\zeta)\bigg|_{\zeta=-\pi/2}^{\zeta=\zeta_{\xx}}&=C_j\cdot\left(\zeta_{\xx}+\frac{\pi}{2}\right)+\frac{|\xi_2(\xx)|\cdot\tilde{P}_j(\xi_1(\xx),A^2(\xx))}{A(\xx)^j}.\label{eq:ij2}
\end{align}
Here $\tilde{P}_j$ is another polynomial.

From \eqref{eq:ijk}, we obtain
\begin{equation}\label{eq:idiff}
\begin{aligned}
I_{jk}(\xx)&=\left(I^{j+2}(\zeta)I^k(\bar{\zeta})-I^k(\zeta)I^{j+2}(\bar{\zeta})\right)\bigg|_{\bar{\zeta}=\zeta_{\xx}}^{\bar{\zeta}=\pi/2}\;\bigg|_{\zeta=-\pi/2}^{\zeta=\zeta_{\xx}}\\
&+\left(I^{k+2}(\zeta)I^j(\bar{\zeta})-I^{j}(\zeta)I^{k+2}(\bar{\zeta})\right)\bigg|_{\bar{\zeta}=\zeta_{\xx}}^{\bar{\zeta}=\pi/2}\;\bigg|_{\zeta=-\pi/2}^{\zeta=\zeta_{\xx}}.
\end{aligned}
\end{equation}
It follows from \eqref{eq:ij1}, \eqref{eq:ij2} and \eqref{eq:idiff} that
\[I_{jk}(\xx)=\frac{\pi |\xi_2(\xx)|}{A(\xx)^{j+k+2}}(C_{j+2}\cdot \tilde{P}_k-C_k\cdot\tilde{P}_{j+2}+C_{k+2}\cdot\tilde{P}_j-C_j\cdot\tilde{P}_{k+2})\left(\xi_1(\xx),A^2(\xx)\right).\]
Therefore $A(\xx)^{j+k+2}\sg(\xi_2(\xx))I_{jk}(\xx)$ is smooth.

Step 8: We now show that the function $A(\xx)e^{2\pi is(\xx)}$ can be smoothly extended to $C_+$.

It follows from \eqref{eq:si} and \eqref{eq:si2} that, in a neighborhood of $C_+$,
\[s(\xx)=\sg(\xi_2(\xx))\frac{1}{2 G\circ v(\xx)}\int_{\zeta_{\xx}}^{\pi/2}\frac{K\circ W^{-1}}{W'\circ W^{-1}}(A(\xx)\sin\zeta) d\zeta.
\]
We recall from \eqref{eq:sumGv} that $G(M)=\frac{\pi K(\bar{u})}{W'(\bar{u})}$. Now arguing as in the previous step and using \eqref{eq:ij1}, for a fixed $n$, we obtain:
\begin{align*}
s(\xx)=\sg(\xi_2(\xx))\left(\frac{1}{2\pi}\left(\frac{\pi}{2}-\zeta_{\xx}\right)\left(1+P^n(A(\xx)^2)\right)\right)+\xi_2(\xx)\tilde{P}^n(\xx)+O(A^{2n+2}(\xx)).
\end{align*}
Here $P^n$ is a degree $n$ polynomial and $\tilde{P}^n$ is a smooth function. Hence
\begin{align*}e^{2\pi i s(\xx)}&=\left(1+\underline{P}^n(\xx)+O\left(A^{2n+2}(\xx)\right)\right)\cdot\left(\cos\left(\sg(\xi_2(\xx))\left(\frac{\pi}{2}-\zeta_{\xx}\right)\right),\sin\left(\sg(\xi_2(\xx))\left(\frac{\pi}{2}-\zeta_{\xx}\right)\right)\right)\\
&=\left(1+\underline{P}^n(\xx)+O\left(A^{2n+2}(\xx)\right)\right)\cdot\left(\frac{\xi_1(\xx)}{A(\xx)},\frac{\xi_2(\xx)}{A(\xx)}\right).\end{align*}
where $\underline{P}^n$ is a smooth function which vanishes along $C_+$. It follows that
\begin{equation*}
A(\xx)e^{2\pi i s(\xx)}=(\xi_1(\xx),\xi_2(\xx))+S(\xx)+O(A^{2n+2}(\xx)).
\end{equation*}
Here $S$ and $O$ take values on $\C$ and $S$ is a smooth function vanishing along $C_+$.
Therefore we can smoothly extend $A(\xx)e^{2\pi i s(\xx)}$ to $C_+$.

Step 9: We now show that $\widehat{\Phi}$ can be smoothly extended to $C_+$.

We define $\Phi$ by substituting $\widehat{\Phi}$ by $\Phi$ in \eqref{eq:defpsi} and by considering the extensions of $\rho_1$, $\rho_2$ and $\phi_2$ to $C_+$. We note that $\phi_1$ is not defined on $C_+$, but $\Phi(\xx)$ is still well-defined on $C_+$, because $\rho_1|_{C_+}=0$. Since $\rho_2|_{C_+}=M$, it follows that $\sqrt{\rho_2(\xx)}$ is smooth on $C_+$ and hence, by Step 7, $\sqrt{\rho_2(\xx)}e^{2\pi i\phi_2(\xx)}=\sqrt{\rho_2(\xx)}e^{i\psi(\xx)}$ is smooth. Now from Definition \ref{def:coord} and \eqref{eq:h} we can write \[\sqrt{\rho_1(\xx)}e^{2\pi i \phi_1(\xx)}=A(\xx)h(\xx)e^{2\pi i s(\xx)}e^{-i\psi(\xx)}=(A(\xx)e^{2\pi i s(\xx)}) h(\xx)e^{-i\psi(\xx)}.\]
So, by Steps 7 and 8, the function $\sqrt{\rho_1(\xx)}e^{2\pi i\phi_1(\xx)}$ is smooth on $C_+$.
Therefore $\Phi$ is smooth on $C_+$.
\end{proof}

\subsection{The symplectomorphism}

We need one more lemma before proving Theorem \ref{thm:symp}. For $0<\epsilon<1$, let $\Omega_\epsilon$ be the region of the first quadrant of $\R^2$ bounded by the coordinate axes and the image of the function $(\rho_1,\rho_2)$. We now let $X_\epsilon$ be the toric domain $X_{\Omega_\epsilon}$.
\begin{lemma}\label{lem:omega}
For $\epsilon_1<\epsilon_2$, we have $X_{\epsilon_1}\supset X_{\epsilon_2}$. Moreover $\bigcup_\epsilon X_{\epsilon}=\interior(X_0)$.
\end{lemma}
\begin{proof}
It follows from the definition of $\rho_1$ and $\rho_2$ that the boundary of $\Omega_\epsilon$ is parametrized by \begin{equation}\label{eq:par2} \left(G(v)-\alpha(v) v,G(v)+(2\pi-\alpha(v))v\right),\end{equation} for $v\in[-M,M]$. We recall that $G$, $\alpha$ and $M$ depend on $\epsilon$. Let $\sigma_\epsilon(v)=G(v)-\alpha(v) v$. We claim that $\sigma_\epsilon(v)$ increases as $\epsilon$ decreases for a fixed $v$. We first show that for $v>0$. After a change of variables, it follows from \eqref{eq:sumin} that for $v>0$,
\begin{equation}
\sigma_\epsilon(v)=\int_{u_1}^{u_0}\frac{\sqrt{F(u)-v^2}}{u}\,du.\label{eq:int}
\end{equation}
Here $u_0$ and $u_1$ are the maximum and minimum values of $|\qq(t)|^2$ where $(\qq(t),\pp(t))$ is a Reeb trajectory with $v=v(\qq(t),\pp(t))$. Since $v^2=u_0(1-\epsilon U(u_0))=u_1(1-\epsilon U(u_1))$, we conclude that if we fix $v$ and decrease $\epsilon$, the value of $u_0$ is increased and the value of $u_1$ is decreased. Moreover, for a fixed $u$, the value of $F(u)=u(1-\epsilon U(u))$ increases as $\epsilon$ decreases. Therefore as $\epsilon$ decreases, the integrand in \eqref{eq:int} increases and so does the interval of integration. Thus for a fixed $v>0$, $\sigma_\epsilon(v)$ increases as $\epsilon$ decreases. Now since $\sigma_\epsilon(v)=\sigma_\epsilon(-v)-2\pi v$, it follows that $\sigma_\epsilon(v)$ increases as $\epsilon$ decreases for a fixed $v<0$. So for $v<0$, $\sigma_\epsilon(v)$ also increases as $\epsilon$ decreases. Finally, for $v=0$, it follows from \eqref{udot} and from the definition of $K$ that
\begin{equation}\label{eq:G0}\sigma_\epsilon(0)=G(0)=\int_0^{u_0}\frac{K(u)}{\sqrt{F(u)}}\,du=\int_0^{u_0}\frac{\sqrt{F(u)}}{u}\,du.\end{equation}
So $\sigma_\epsilon(0)$ increases as $\epsilon$ decreases. Hence for fixed $v$, both coordinates of \eqref{eq:par2} increase.
Therefore for $\epsilon_1<\epsilon_2$, we have $X_{\epsilon_1}\supset X_{\epsilon_2}$.

We will now show that $\bigcup_\epsilon X_{\epsilon}=\interior(X_0)$. Recall that $M=\sqrt{F(\bar{u})}$, where $\bar{u}$ is the critical point of $F$. It follows from a simple calculation that as $\epsilon$ decreases, the values of $\bar{u}$ and $M$ increase to $1$. Now let $V(\alpha_0)=\cos\left(\frac{\alpha_0}{2}\right)$. Note that $V$ is a diffeomorphism between $(0,2\pi)$ and $(-1,1)$. Moreover for every $\alpha_0\in(0,2\pi)$, the number $\sigma_\epsilon(V(\alpha_0))$ is defined for $\epsilon$ small enough. We claim that 
\begin{equation}\label{eq:claim}\lim_{\epsilon\to 0}\sigma_\epsilon(V(\alpha_0))=2\sin\left(\frac{\alpha_0}{2}\right)-\alpha\cos\left(\frac{\alpha_0}{2}\right),\end{equation} for every $\alpha_0\in(0,2\pi)$.
From this it follows that $\bigcup_\epsilon X_{\epsilon}=\interior(X_0)$.

We will now prove \eqref{eq:claim}. First we assume that $\alpha_0\in(0,\pi)$. So $V(\alpha_0)\in(0,1)$. Let $v=V(\alpha_0)$ and let $u_0,u_1$ be as above. Since $v^2=u_1(1-\epsilon U(u_1))$ and $U(u_1)$ is bounded, it follows that $u_1\to v^2$ as $\epsilon\to 0$. Moreover $u_0\to 1$ as $\epsilon\to 0$, since $u_0>\bar{u}$ and $\bar{u}\to 1$. We observe that the function $\frac{\sqrt{F(u)-v^2}}{u}$ is a continuous function of $(\epsilon,u)$ for $0\le\epsilon\le 1/2$ and $u_0\le u\le u_1$. Note that $u_0$ and $u_1$ depend on $\epsilon$, so the region described above is not a rectangle. It follows that we can take the limit as $\epsilon\to 0$ in \eqref{eq:int} and we obtain
\begin{align*}
\lim_{\epsilon\to 0}\sigma_\epsilon(v)&=\int_{v^2}^1\frac{\sqrt{u-v^2}}{u}\,du=2\sqrt{1-v^2}-2v\arctan\frac{\sqrt{1-v^2}}{v}\\
&=2\sin\left(\frac{\alpha_0}{2}\right)-\alpha_0 \cos\left(\frac{\alpha_0}{2}\right).
\end{align*}
Hence \eqref{eq:claim} holds for $\alpha_0\in(0,\pi)$. For $\alpha_0\in(\pi,2\pi)$, we have $v=V(\alpha_0)\in(-1,0)$.
Since $G(-v)=G(v)$ and $\alpha(-v)=2\pi-\alpha$, we obtain
\begin{align*}\lim_{\epsilon\to 0}\sigma_\epsilon(v)&=\lim_{\epsilon\to 0}(G(-v)-\alpha(-v))(-v)-2\pi v)\\
&=2\sin\frac{\alpha_0(-v)}{2}-\alpha_0(-v)\cos\frac{\alpha_0(-v)}{2}-2\pi\cos\frac{\alpha_0(v)}{2}\\
&=2\sin\frac{2\pi-\alpha_0(v)}{2}-(2\pi-\alpha_0(v))\cos\frac{2\pi-\alpha_0(v)}{2}-2\pi\cos\frac{\alpha_0(v)}{2}\\
&=2\sin\frac{\alpha_0(v)}{2}-\alpha_0(v)\cos\frac{\alpha_0(v)}{2}.
\end{align*}
So \eqref{eq:claim} holds for $\alpha_0\in(\pi,2\pi)$. For $\alpha_0=\pi$, we claim that we can still take the limit as $\epsilon\to 1$ inside the integral \eqref{eq:G0}. In fact, since $\frac{\sqrt{F(u)}}{u}$ is continuous in $(\epsilon,u)$ for $\epsilon$ away from $0$, it is enough to show that
\begin{equation}\label{eq:limeps}\lim_{\epsilon\to 0}\int_0^{1/2} \left(\frac{\sqrt{F(u)}}{u}-\frac{1}{\sqrt{u}}\right)\,du=0.\end{equation}
To prove that, we first observe that
\[\int_0^{1/2}\left(\frac{\sqrt{F(u)}}{u}-\frac{1}{\sqrt{u}}\right)\,du=\int_0^{1/2}\frac{\epsilon U(u)}{\sqrt{u}(\sqrt{1-\epsilon U(u)}+1)}\,du.\]
Since the function $\frac{U(u)}{\sqrt{1-\epsilon U(u)}+1}$ is uniformly bounded in $[0,1/2]$ for $\epsilon$ sufficiently small and $\int_0^{1/2}\frac{1}{\sqrt{u}}\,du$ converges, it follows that \eqref{eq:limeps} holds and therefore
\[\lim_{\epsilon\to 0} \sigma_\epsilon(V(\pi))=\lim_{\epsilon\to 0} \sigma_\epsilon(0)=\int_0^1\frac{1}{\sqrt{u}}\,du=2.\]
So \eqref{eq:claim} holds for $\alpha_0=\pi$.
\end{proof}

\begin{proof}[Proof of Theorem \ref{thm:symp}]
It follows from Lemma \ref{lem:Phi} that $\Phi$ is a strict contactomorphism from $(\partial P_\epsilon,\lambda)$ to $(\partial X_{\epsilon},\lambda)$. So $\Phi$ induces a symplectomorphism
 \begin{equation}\label{eq:symplecto}\left((-\delta,\infty)\times \partial P_\epsilon, d(e^t\lambda|_{\partial P_\epsilon})\right)\cong \left((-\delta,\infty)\times \partial X_{\epsilon}, d(e^t\lambda|_{\partial X_{\epsilon}})\right),\end{equation} for $\delta>0$. We recall that the flow of the Liouville vector field induces symplectomorphisms between neighborhoods of $\R^4\setminus P_\epsilon$ and $\R^4\setminus X_{\epsilon}$ and the two manifolds in \eqref{eq:symplecto} for small $\delta$, respectively. Therefore we obtain a symplectomorphism between neighborhoods of $\R^4\setminus P_\epsilon$ and $\R^4\setminus X_{\epsilon}$. Therefore, by Gromov-McDuff \cite[Theorem 9.4.2]{ms} there is a symplectomorphism of $\R^4$ that maps $P_{\epsilon}$ to $X_{\epsilon}$ and whose restriction to $\partial P_{\epsilon}$ is $\Phi$. We observe that for varying $\epsilon$, these symplectomorphisms define a smooth family of symplectic embeddings $\Phi_\epsilon:P_\epsilon\hookrightarrow \C^2$, whose images are $X_{\epsilon}$.

By Lemma \ref{lem:omega}, the domains $X_{\epsilon}$ are nested and \[\bigcup_{\epsilon}X_{\epsilon}=\interior(X_0).\]
We would like to take the limit of the maps $\Phi_\epsilon$, but we cannot do that in general, because the maps $\Phi_\epsilon$ are not nested.
Instead we use an argument from \cite{mcduffblowup}. Let $\{\epsilon_n\}$ be a decreasing sequence of real numbers converging to $0$.
Since $\Phi_{\epsilon_{n-1}}(P_{\epsilon_{n-1}})\subset \Phi_{\epsilon_{n}}(P_{\epsilon_{n}})$, we have an isotopy of symplectic embeddings $\Phi_t\circ \Phi_{\epsilon_{n+1}}^{-1}:P_{\epsilon_{n-1}}\to P_{\epsilon_{n}}$ for $t\in[\epsilon_{n},\epsilon_{n-1}]$. By the symplectic isotopy extension theorem, we can extend it to an isotopy of symplectomorphisms of $P_{\epsilon_n}$ for $t\in[\epsilon_{n},\epsilon_{n-1}]$. We compose this isotopy with $\Phi_{\epsilon_{n}}$ and take $t=\epsilon_{n-1}$ to obtain another symplectic embedding, denoted by $\Psi_{n}: P_{\epsilon_{n}}\to \mathbb{C}^2$, with image $X_{\epsilon_n}$ such that
\begin{itemize}
\item $\Psi_{n}|_{P_{\epsilon_{n-1}}}=\Phi_{\epsilon_{n-1}}$,
\item $\Psi_{n}=\Phi_{\epsilon_{n}}$ in a neighborhood of $\partial P_{\epsilon_{n}}$.
\end{itemize}
So we can define the embedding $\Psi_{\infty}:\interior(P_L)\to\mathbb{C}^2$ by $\Psi_\infty(x)=\Psi_n(x)$ for $x\in P_{\epsilon_n}\setminus P_{\epsilon_{n-1}}$. It follows from the properties above that $\Psi_\infty$ is smooth and that its image is
$\bigcup_\epsilon X_{\epsilon}=\interior(X_0)$.
\end{proof}

\section{Embedding $\interior(X_0)$ into $E(4,3\sqrt{3})$}\label{sec:emb}

In this section, we will prove Proposition \ref{prop:emb2}. Before that, we recall some constructions for toric domains. Let $T(a,b)$ be a triangle with vertices at $(0,0)$, $(a,0)$ and $(0,b)$. Then $X_{T(a,b)}=E(a,b)$. Let $T'(a,b)$ be a translation of $T(a,b)$ that is contained in the interior of $\R_{\ge 0}^2$. We let $E'(a,b)=X_{T'(a,b)}$. We note that a different choice of $T'(a,b)$ induces a symplectomorphic $X_{T'(a,b)}$. Therefore we do not need to specify the translation as long as $T'(a,b)$ does not intersect the coordinate axes. We will write $B'(a):=E'(a,a)$, $T'(a)=T'(a,a)$ and $T(a)=T(a,a)$.

We now recall two facts about toric domains that will be important in what follows.
\begin{lemma}\label{lem:sl2z}
If $\Omega\subset\R_+^2$ does not intersect the axes and $A\in SL(2,\Z)$, then $X_\Omega$ is symplectomorphic to $X_{A\cdot \Omega}$.
\end{lemma}
\begin{lemma}\label{lem:tr}
For every $\epsilon>0$, there exists a symplectic embedding \[E(a,b)\hookrightarrow (1+\epsilon)E'(a,b).\]
\end{lemma}
Lemma \ref{lem:sl2z} follows from a standard calculation and Lemma \ref{lem:tr} is a result of Traynor \cite{traynor}, see also \cite{lmt}. We also observe that $E'(a,b)\subset E(a,b)$.
Now let $X_\Omega$ be a concave toric domain and let $w_1\ge w_2\ge \dots$ be its weight sequence as defined in \S\ref{sec:pack}. This procedure also determines a decomposition of $\Omega$ into (nondisjoint) regions, which are affine equivalent to triangles. We will now recall how to obtain the embedding \eqref{eq:pack}. Let $\epsilon>0$. By Lemma \ref{lem:tr}, $B(w_i)\hookrightarrow\left(1+\frac{\epsilon}{2}\right)B'(w_i)$ for all $i$.

For every $i$, the interior of $T'(w_i)$ is contained in a region of $\Omega$ after multiplying by a certain number of matrices in \[\left\{\left[\begin{array}{cc}1 &0\\1&1\end{array}\right]^{-1},\left[\begin{array}{cc}1&1\\0&1\end{array}\right]^{-1}\right\}\subset SL(2,\Z).\]
So $(1+\frac{\epsilon}{2})B'(w_i)\hookrightarrow (1+\epsilon)X_\Omega$, for every $i$ and the images of these embeddings are pairwise disjoint. Therefore we obtain the embedding \eqref{eq:pack}.

We now state a simple lemma that will be useful in the proof of Proposition \ref{prop:emb2} and whose proof is a straightforward consequence of \eqref{eq:disj}.
\begin{lemma}\label{lemma:sqcup}
Let $A_1,A_2,B_1,B_2$ be symplectic 4-manifolds such that $c_k(A_i)=c_k(B_i)$ for all $k$, for $i=1,2$. Then for all $k$
\[c_k(A_1\sqcup A_2)= c_k(B_1\sqcup B_2).\]
\end{lemma}

\begin{proof}[Proof of Proposition \ref{prop:emb2}]
\begin{figure}
\centering
\begin{subfigure}[b]{0.27\textwidth}
\centering
\def\svgwidth{0.9\textwidth}
\begingroup%
  \makeatletter%
  \providecommand\color[2][]{%
    \errmessage{(Inkscape) Color is used for the text in Inkscape, but the package 'color.sty' is not loaded}%
    \renewcommand\color[2][]{}%
  }%
  \providecommand\transparent[1]{%
    \errmessage{(Inkscape) Transparency is used (non-zero) for the text in Inkscape, but the package 'transparent.sty' is not loaded}%
    \renewcommand\transparent[1]{}%
  }%
  \providecommand\rotatebox[2]{#2}%
  \ifx\svgwidth\undefined%
    \setlength{\unitlength}{192.484375bp}%
    \ifx\svgscale\undefined%
      \relax%
    \else%
      \setlength{\unitlength}{\unitlength * \real{\svgscale}}%
    \fi%
  \else%
    \setlength{\unitlength}{\svgwidth}%
  \fi%
  \global\let\svgwidth\undefined%
  \global\let\svgscale\undefined%
  \makeatother%
  \begin{picture}(1,0.66498904)%
    \put(0,0){\includegraphics[width=\unitlength]{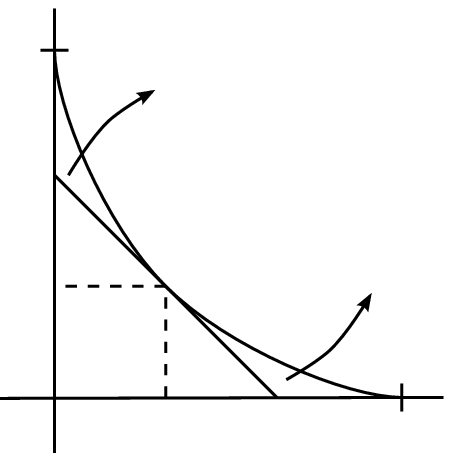}}%
\put(0.4,0.8){\color[rgb]{0,0,0}\makebox(0,0)[lb]{\smash{$\Omega_3$}}}%
    \put(0.8,0.4){\color[rgb]{0,0,0}\makebox(0,0)[lb]{\smash{$\Omega_2$}}}%
    \put(0.32,0.02){\color[rgb]{0,0,0}\makebox(0,0)[lb]{\smash{$x_1$}}}%
    \put(0.57,0.02){\color[rgb]{0,0,0}\makebox(0,0)[lb]{\smash{$w_1$}}}%
    \put(0.82,0.02){\color[rgb]{0,0,0}\makebox(0,0)[lb]{\smash{$x_2$}}}%
    \put(0,0.35){\color[rgb]{0,0,0}\makebox(0,0)[lb]{\smash{$y_1$}}}%
    \put(-0.02,0.6){\color[rgb]{0,0,0}\makebox(0,0)[lb]{\smash{$w_1$}}}%
    \put(-0.02,0.9){\color[rgb]{0,0,0}\makebox(0,0)[lb]{\smash{$y_2$}}}
   
  \end{picture}%
\endgroup%
\caption{The region $\Omega_0$:\\
$x_1=2$,\\
$w_1=4$,\\
$x_2=2\pi$,\\
$y_1=2$,\\
$y_2=2\pi$.}
\end{subfigure}
\quad
\begin{subfigure}[b]{0.27\textwidth}
\centering
\def\svgwidth{0.9\textwidth}
\begingroup%
  \makeatletter%
  \providecommand\color[2][]{%
    \errmessage{(Inkscape) Color is used for the text in Inkscape, but the package 'color.sty' is not loaded}%
    \renewcommand\color[2][]{}%
  }%
  \providecommand\transparent[1]{%
    \errmessage{(Inkscape) Transparency is used (non-zero) for the text in Inkscape, but the package 'transparent.sty' is not loaded}%
    \renewcommand\transparent[1]{}%
  }%
  \providecommand\rotatebox[2]{#2}%
  \ifx\svgwidth\undefined%
    \setlength{\unitlength}{135.46875bp}%
    \ifx\svgscale\undefined%
      \relax%
    \else%
      \setlength{\unitlength}{\unitlength * \real{\svgscale}}%
    \fi%
  \else%
    \setlength{\unitlength}{\svgwidth}%
  \fi%
  \global\let\svgwidth\undefined%
  \global\let\svgscale\undefined%
  \makeatother%
  \begin{picture}(1,0.91534025)%
    \put(0,0){\includegraphics[width=\unitlength]{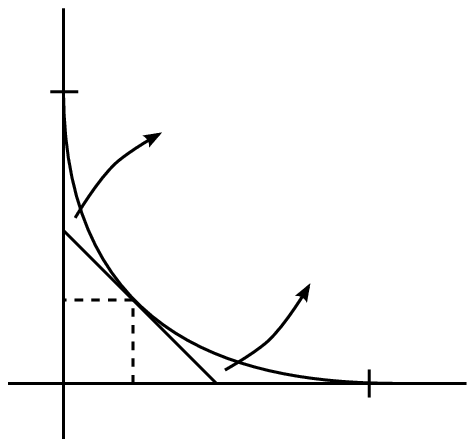}}%
    \put(0.75,0){\color[rgb]{0,0,0}\makebox(0,0)[lb]{\smash{$x_2'$}}}%
    \put(0.4094579,0){\color[rgb]{0,0,0}\makebox(0,0)[lb]{\smash{$w_2$}}}%
    \put(-0.02,0.73825332){\color[rgb]{0,0,0}\makebox(0,0)[lb]{\smash{$y_2'$}}}%
    \put(-0.02,0.44298227){\color[rgb]{0,0,0}\makebox(0,0)[lb]{\smash{$w_2$}}}%
    \put(0.23229527,0){\color[rgb]{0,0,0}\makebox(0,0)[lb]{\smash{$x_1'$}}}%
    \put(-0.01,0.29534674){\color[rgb]{0,0,0}\makebox(0,0)[lb]{\smash{$y_1'$}}}%
    \put(0.37730816,0.62966389){\color[rgb]{0,0,0}\makebox(0,0)[lb]{\smash{$\Omega_5$}}}%
    \put(0.64127604,0.37){\color[rgb]{0,0,0}\makebox(0,0)[lb]{\smash{$\Omega_4$}}}%
  \end{picture}%
\endgroup%
\caption{The region $\Omega_2$:\\
$x_1'=2\sqrt{3}+\pi/3-4$,\\ $w_2=3\sqrt{3}-4$,\\ $x_2'=2\pi-4$,\\$y_1'=\sqrt{3}-\pi/3$,\\ $y_2'=2$}
\end{subfigure}
\quad
\begin{subfigure}[b]{0.27\textwidth}
\centering
\def\svgwidth{0.9\textwidth}
\begingroup%
  \makeatletter%
  \providecommand\color[2][]{%
    \errmessage{(Inkscape) Color is used for the text in Inkscape, but the package 'color.sty' is not loaded}%
    \renewcommand\color[2][]{}%
  }%
  \providecommand\transparent[1]{%
    \errmessage{(Inkscape) Transparency is used (non-zero) for the text in Inkscape, but the package 'transparent.sty' is not loaded}%
    \renewcommand\transparent[1]{}%
  }%
  \providecommand\rotatebox[2]{#2}%
  \ifx\svgwidth\undefined%
    \setlength{\unitlength}{126.62827148bp}%
    \ifx\svgscale\undefined%
      \relax%
    \else%
      \setlength{\unitlength}{\unitlength * \real{\svgscale}}%
    \fi%
  \else%
    \setlength{\unitlength}{\svgwidth}%
  \fi%
  \global\let\svgwidth\undefined%
  \global\let\svgscale\undefined%
  \makeatother%
  \begin{picture}(1,0.66375383)%
    \put(0,0){\includegraphics[width=\unitlength]{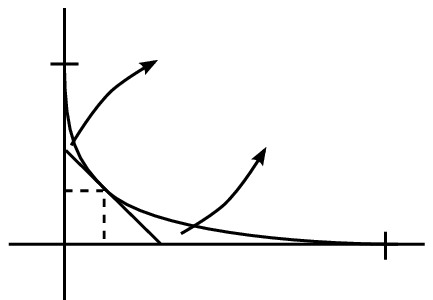}}%
    \put(0.8442291,0.02){\color[rgb]{0,0,0}\makebox(0,0)[lb]{\smash{$x_2''$}}}%
    \put(0.35,0.02){\color[rgb]{0,0,0}\makebox(0,0)[lb]{\smash{$w_4$}}}%
    \put(-0.01,0.53748072){\color[rgb]{0,0,0}\makebox(0,0)[lb]{\smash{$y_2''$}}}%
    \put(-0.01,0.36){\color[rgb]{0,0,0}\makebox(0,0)[lb]{\smash{$w_4$}}}%
    \put(0.21245865,0.02){\color[rgb]{0,0,0}\makebox(0,0)[lb]{\smash{$x_1''$}}}%
    \put(0,0.25){\color[rgb]{0,0,0}\makebox(0,0)[lb]{\smash{$y_1''$}}}%
    \put(0.40198978,0.53748072){\color[rgb]{0,0,0}\makebox(0,0)[lb]{\smash{$\Omega_7$}}}%
    \put(0.59152092,0.3795381){\color[rgb]{0,0,0}\makebox(0,0)[lb]{\smash{$\Omega_6$}}}%
  \end{picture}%
\endgroup%
\caption{The region $\Omega_4$:\\
$x_1''=3\sqrt{2}+\pi\sqrt{2}/4-3\sqrt{3}$,\\ $w_4=4\sqrt{2}-3\sqrt{3}$,\\ $x_2''=2\pi-3\sqrt{3}$,\\$y_1''=\sqrt{2}-\pi\sqrt{2}/4$,\\ $y_2''=\sqrt{3}-\pi/3$}
\end{subfigure}
\caption{The first weights of $\Omega_0$}
\label{fig:w}
\end{figure}
By Theorem \ref{thm:cc}, it is enough to show that for all $k$,
\begin{equation}\label{eq:inc}
c_k(X_0)\le c_k(E(4,3\sqrt{3})).
\end{equation}
We recall that $\Omega_0$ is the region bounded by the coordinate axes and the curve \eqref{eq:param}. Let $w_1\ge w_2\ge w_3\ge \dots$ be the weight sequence of $\Omega_0$. 
 It follows from an easy calculation that $w_1= 4$, see Figure \ref{fig:w}(a). We obtain domains $\Omega_2$ and $\Omega_3$ by applying an affine transformation to $\Omega_0\setminus T(4)$ as explained in \S\ref{sec:pack}. Moreover $\Omega_2$ and $\Omega_3$ are equal since the curve \eqref{eq:param} is symmetric with respect to the reflection across the line $y=x$. In particular $w_2=w_3$ and all the following weights come in pairs. Continuing the calculation, we obtain $w_2=w_3=3\sqrt{3}-4$, see Figure \ref{fig:w}(b). Let $\Omega_4$ and $\Omega_5$ be the next regions in this process obtained from $\Omega_2$, see Figure \ref{fig:w}(b). We now observe that the next weight coming from $\Omega_4$ is $4\sqrt{2}-3\sqrt{3}$. It turns out that $w_4=w_5=4\sqrt{2}-3\sqrt{3}$, but we do not need this fact for what follows. Let $\Omega_6$ and $\Omega_7$ be the regions obtained from $\Omega_4$, see Figure \ref{fig:w}(c). We also let $\widetilde{\Omega}_5$, $\widetilde{\Omega}_6$ and $\widetilde{\Omega}_7$ be the respective regions obtained from $\Omega_3$. We note that $\Omega_i$ and $\widetilde{\Omega}_i$ are equal for $i=5,6,7$. It follows from Theorem \ref{thm:ccfhr} and Lemma \ref{lemma:sqcup} that
\begin{equation}\label{eq:part}
\begin{aligned}
c_k(X_0)&=c_k\left(\coprod_{i=1}^{\infty} B(w_i)\right)\\
&=c_k\left( B(4)\sqcup\coprod_{j=1}^2 B(3\sqrt{3}-4)\sqcup\coprod_{j=1}^2 B(4\sqrt{2}-3\sqrt{3})\sqcup\coprod_{j=5}^7\left( X_{\Omega_j}\sqcup X_{\widetilde{\Omega}_j}\right)\right).\end{aligned}
\end{equation}
\captionsetup{justification=centering}
\begin{figure}
 \centering
 \def\svgwidth{0.7\textwidth}
 \begingroup%
  \makeatletter%
  \providecommand\color[2][]{%
    \errmessage{(Inkscape) Color is used for the text in Inkscape, but the package 'color.sty' is not loaded}%
    \renewcommand\color[2][]{}%
  }%
  \providecommand\transparent[1]{%
    \errmessage{(Inkscape) Transparency is used (non-zero) for the text in Inkscape, but the package 'transparent.sty' is not loaded}%
    \renewcommand\transparent[1]{}%
  }%
  \providecommand\rotatebox[2]{#2}%
  \ifx\svgwidth\undefined%
    \setlength{\unitlength}{457.15bp}%
    \ifx\svgscale\undefined%
      \relax%
    \else%
      \setlength{\unitlength}{\unitlength * \real{\svgscale}}%
    \fi%
  \else%
    \setlength{\unitlength}{\svgwidth}%
  \fi%
  \global\let\svgwidth\undefined%
  \global\let\svgscale\undefined%
  \makeatother%
  \begin{picture}(1,0.74083999)%
    \put(0,0){\includegraphics[width=\unitlength]{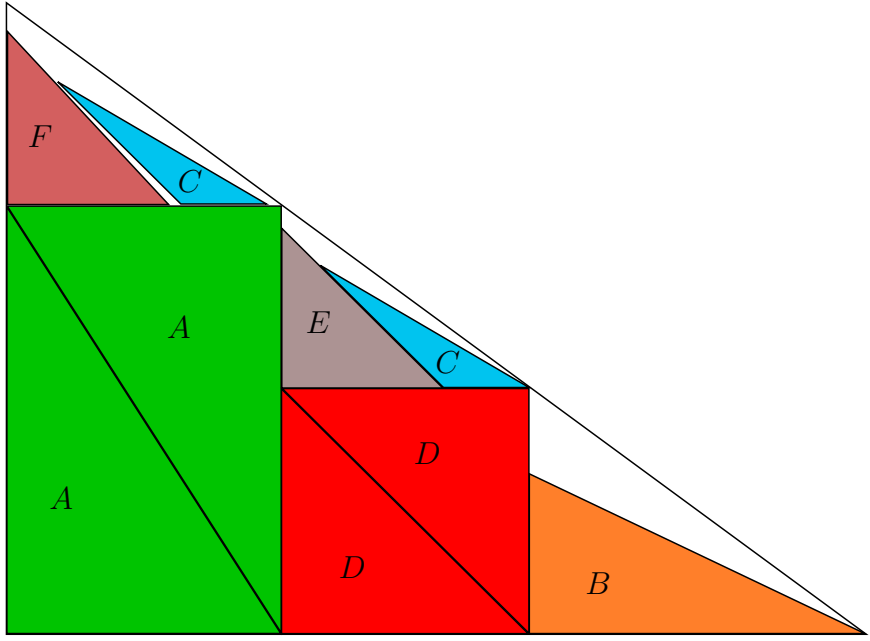}}%
    \put(0.05048051,0.1500421){\color[rgb]{0,0,0}\makebox(0,0)[lb]{\smash{$A$}}}%
    \put(0.1868957,0.34846416){\color[rgb]{0,0,0}\makebox(0,0)[lb]{\smash{$A$}}}%
    \put(0.3852471,0.06964782){\color[rgb]{0,0,0}\makebox(0,0)[lb]{\smash{$D$}}}%
    \put(0.47178116,0.20246748){\color[rgb]{0,0,0}\makebox(0,0)[lb]{\smash{$D$}}}%
    \put(0.19929708,0.51898321){\color[rgb]{0,0,0}\makebox(0,0)[lb]{\smash{$C$}}}%
    \put(0.49693023,0.30815967){\color[rgb]{0,0,0}\makebox(0,0)[lb]{\smash{$C$}}}%
    \put(0.34701112,0.35541139){\color[rgb]{0,0,0}\makebox(0,0)[lb]{\smash{$E$}}}%
    \put(0.02567775,0.57168907){\color[rgb]{0,0,0}\makebox(0,0)[lb]{\smash{$F$}}}%
    \put(0.67054954,0.05083102){\color[rgb]{0,0,0}\makebox(0,0)[lb]{\smash{$B$}}}%
  \end{picture}%
\endgroup%
\caption{Fitting the triangles into $T'(1.607,1.19)$\\
$A=T'(0.512,0.804)$, $B=T'(0.627,0.304)$, $C=T'(0.16,0.23)$,\\
$D=T'(0.464)$, $E=T'(0.304)$, $F=T'(0.323,0.304)$.}
\label{fig:final}
 \end{figure}
We now observe that the first three terms in the weight decomposition of $T(4,3\sqrt{3})$ are $4,3\sqrt{3}-4,3\sqrt{3}-4$. After three steps, the concave toric domain obtained by the procedure explained in \S\ref{sec:pack} is $E(12-6\sqrt{3},3\sqrt{3}-4)$. Again, by Theorem \ref{thm:ccfhr} and Lemma \ref{lemma:sqcup}, it follows that
\begin{equation}\label{eq:part2}
c_k(E(4,3\sqrt{3}))=c_k\left(B(4)\sqcup\coprod_{j=1}^2 B(3\sqrt{3}-4)\sqcup E(12-6\sqrt{3},3\sqrt{3}-4)\right).
\end{equation} 
From \eqref{eq:part} and \eqref{eq:part2} we conclude that, in order to prove \eqref{eq:inc}, it suffices to show that
\begin{equation}\label{eq:part3}
c_k\left(\coprod_{j=1}^2 B(4\sqrt{2}-3\sqrt{3})\sqcup\coprod_{j=5}^7\left( X_{\Omega_j}\sqcup X_{\widetilde{\Omega}_j}\right)\right)\le c_k\left(E(12-6\sqrt{3},3\sqrt{3}-4)\right).
\end{equation}
 It follows from a simple calculation that
\begin{align*}
\Omega_5,\widetilde{\Omega}_5&\subset T\left(2\sqrt{3}+\frac{\pi}{3}-4,6-3\sqrt{3}\right)\subset T (0.512,0.804) ,\\
\Omega_6,\widetilde{\Omega}_6&\subset  T\left(2\pi-4\sqrt{2},\sqrt{2}-\frac{\pi\sqrt{2}}{4}\right)\subset T (0.627,0.304) ,\\
\Omega_7,\widetilde{\Omega}_7&\subset T\left(3\sqrt{2}-3\sqrt{3}+\frac{\pi\sqrt{2}}{4},4\sqrt{3}-4\sqrt{2}-\frac{\pi}{3}\right)\subset T (0.16,0.23),\\
T(4\sqrt{2}&-3\sqrt{3})\subset T(0.461)\subset T(0.464),\\
T(12-&6\sqrt{3},3\sqrt{3}-4)\supset T(1.607,1.19).
\end{align*}
From Theorem \ref{thm:ccfhr} and Lemma \ref{lemma:sqcup}, we deduce that \[c_k(E(0.627,0.304))=c_k(B(0.304)\sqcup E(0.323,0.304)).\] Using Lemma \ref{lemma:sqcup} and the calculations above, we conclude that \eqref{eq:part3} follows from the existence of the embedding:
\begin{equation}\label{eq:emb2}
\begin{aligned}
 \coprod_{j=1}^2&\left(E'(0.512,0.804)\sqcup B'(0.464)  \sqcup E'(0.16,0.23)\right)\\ &\sqcup  E'(0.627,0.304)  \sqcup B'(0.304)\sqcup E'(0.323,0.304)\hookrightarrow  E'(1.607,1.19).
 \end{aligned}
 \end{equation}
The embedding \eqref{eq:emb2} can be constructed by hand. In fact, we can fit triangles of the appropriate size corresponding to each domain in the left hand side of \eqref{eq:emb2} into $T'(1.607,1.19)$, see Figure \ref{fig:final}. Note that it is crucial to check that the two rectangles obtained by taking two copies of $A$ and $D$, respectively, do fit into $T'(1.607,1.19)$, which follows from a simple computation. Therefore the disjoint union of the induced toric domains embed into $E'(1.607,1.19)$. So we have proved \eqref{eq:emb2} and hence \eqref{eq:part3}.
\end{proof}

\bibliographystyle{siam}
\bibliography{bib}

\end{document}